\documentclass[12pt]{amsart}
\usepackage[notref,notcite]{}

\usepackage{tikz,amsthm,amsmath,amstext,amssymb,amscd,epsfig,euscript, mathrsfs, dsfont,pspicture,multicol,graphpap,graphics,graphicx,times,enumerate,subfig,sidecap,wrapfig,color,pict2e,xcolor}
\usepackage{dsfont}

\usepackage[colorlinks,citecolor=cyan,pagebackref,hypertexnames=false]{hyperref}

\usepackage{enumerate}

 \numberwithin{equation}{section}

\def\bR{{\mathbb{R}}}
\def\R{{\mathbb{R}}}

\def\bZ{{\mathbb{Z}}}

\def\cB{{\mathscr{B}}}
\def\cC{{\mathscr{C}}}
\def\cD{{\mathscr{D}}}

\def\cF{{\mathscr{F}}}
\def\cG{{\mathscr{G}}}
\def\cH{{\mathscr{H}}}
\def\cI{{\mathscr{I}}}

\def\cT{{\mathscr{T}}}

\def\sP{\mathfrak{P}}

\def\one{\mathds{1}}

\def\ve{\varepsilon}

\renewcommand{\d}{{\partial}}
\def\lec{\lesssim}
\def\gec{\gtrsim}

\DeclareMathOperator{\diam}{diam}
\def\dist{\mathop\mathrm{dist}} 						

\newcommand{\ps}[1]{\left( #1 \right)}

\newcommand{\ck}[1]{\left\{#1 \right\}}

\def\XXint#1#2#3{{\setbox0=\hbox{$#1{#2#3}{\int}$ }
\vcenter{\hbox{$#2#3$ }}\kern-.58\wd0}}

\theoremstyle{plain}

\newtheorem*{mainlemma}{Main Lemma}
\newtheorem{theorem}{Theorem}

\newtheorem{corollary}[theorem]{Corollary}

\newtheorem{lemma}[theorem]{Lemma}

\theoremstyle{definition}

\newtheorem{definition}[theorem]{Definition}

\newtheorem{remark}[theorem]{Remark}

\numberwithin{equation}{section}
\numberwithin{theorem}{section}

\def\dF{{\mathcal{F}}}

\def\bR{{\mathbb{R}}}

\def\bZ{{\mathbb{Z}}}

\def\cB{{\mathscr{B}}}
\def\cC{{\mathscr{C}}}
\def\cD{{\mathscr{D}}}

\def\cF{{\mathscr{F}}}
\def\cG{{\mathscr{G}}}
\def\cH{{\mathscr{H}}}
\def\cI{{\mathscr{I}}}

\def\sP{\mathfrak{P}}

\def\one{\mathds{1}}

\newcommand{\wh}{\widehat}

\def\ve{\epsilon} 

\renewcommand{\d}{{\partial}}
\def\lec{\lesssim}
\def\gec{\gtrsim}

\newcommand{\hd}{\mathscr{H}^d}

\newcommand{\hdc}{\mathscr{H}^d_\infty}
\newcommand{\spt}{\mathrm{spt}}
\newcommand{\dt}{\frac{dt}{t}}
\newcommand{\dr}{\frac{dr}{r}}

\newcommand{\chara}{\mathds{1}}

\newcommand{\lesssimt}[1]{\ensuremath{\stackrel{#1}{\lesssim}}}

\newcommand{\Stop}{\mathop\mathrm{\textbf{Stop}}}

\newcommand{\Next}{\mathop\mathrm{\textbf{Next}}}

\newcommand{\Top}{\mathop\mathrm{\textbf{Top}}}

\newcommand{\wt}{\widetilde}

\newcommand{\bad}[1]{\cB^{\text{#1}}}
\newcommand{\good}[1]{\cG^{\text{#1}}}
\newcommand{\Child}{\text{Child}}
\newcommand{\Approx}{\text{Approx}}

  \DeclareFontFamily{U}{mathb}{\hyphenchar\font45} 
\DeclareFontShape{U}{mathb}{m}{n}{
      <5> <6> <7> <8> <9> <10> gen * mathb
      <10.95> mathb10 <12> <14.4> <17.28> <20.74> <24.88> mathb12
      }{}
\DeclareSymbolFont{mathb}{U}{mathb}{m}{n}

\DeclareMathSymbol{\toitself}{3}{mathb}{"FD}  

\makeatletter
\def\@tocline#1#2#3#4#5#6#7{\relax
  \ifnum #1>\c@tocdepth 
  \else
    \par \addpenalty\@secpenalty\addvspace{#2}%
    \begingroup \hyphenpenalty\@M
    \@ifempty{#4}{%
      \@tempdima\csname r@tocindent\number#1\endcsname\relax
    }{%
      \@tempdima#4\relax
    }%
    \parindent\z@ \leftskip#3\relax \advance\leftskip\@tempdima\relax
    \rightskip\@pnumwidth plus4em \parfillskip-\@pnumwidth
    #5\leavevmode\hskip-\@tempdima
      \ifcase #1
       \or\or \hskip 1em \or \hskip 2em \else \hskip 3em \fi%
      #6\nobreak\relax
    \dotfill\hbox to\@pnumwidth{\@tocpagenum{#7}}\par
    \nobreak
    \endgroup
  \fi}
\makeatother

  \def\Bad{\mathrm{Bad}}
\def\Top{\textrm{Top}}
\def\Stop{\textrm{Stop}}
\def\Tree{\textrm{Tree}}
\def\Badsh{\textrm{Bad}}

\begin{document}

\author{Jonas Azzam}
\address{School of Mathematics, University of Edinburgh, JCMB, Kings Buildings,
Mayfield Road, Edinburgh,
EH9 3JZ, Scotland.}
\email{j.azzam ``at" ed.ac.uk}

\author{Michele Villa}
\address{School of Mathematics, University of Edinburgh, JCMB, Kings Buildings,
Mayfield Road, Edinburgh,
EH9 3JZ, Scotland.}
\email{m.villa-2 ``at" sms.ed.ac.uk}


\title[Quantitative Comparisons of Multiscale Geometric Properties]{Quantitative Comparisons of Multiscale Geometric Properties}

\begin{abstract}
We generalize some characterizations of uniformly rectifiable (UR) sets to sets whose Hausdorff content is lower regular (and in particular, do not need to be Ahlfors regular).  For example, David and Semmes showed that, given an Ahlfors $d$-regular set $E$, if we consider the set $\cB$ of surface cubes (in the sense of Christ and David) near which $E$ does not look approximately like a union of planes, then $E$ is UR if and only if $\cB$ satisfies a Carleson packing condition, that is, for any surface cube $R$,
\[
\sum_{Q\subseteq R\atop Q\in \cB} (\diam Q)^{d} \lec (\diam R)^{d}.\] 
We show that, for lower content regular sets that aren't necessarily Ahlfors regular, if $\beta_{E}(R)$ denotes the square sum of $\beta$-numbers over subcubes of $R$ as in the Traveling Salesman Theorem for higher dimensional sets \cite{AS18}, then
\[
\cH^{d}(R)+\sum_{Q\subseteq R\atop Q\in \cB} (\diam Q)^{d}\sim  \beta_{E}(R).
\]
We prove similar results for other uniform rectifiability critera, such as the Local Symmetry, Local Convexity, and Generalized Weak Exterior Convexity conditions.

En route, we show how to construct a corona decomposition of any lower content regular set by Ahlfors regular sets, similar to the classical corona decomposition of UR sets by Lipschitz graphs developed by David and Semmes.

\end{abstract}

\maketitle

\tableofcontents

\section{Introduction}

\subsection{Background}
A set $E\subseteq \R^{n}$ is said to be {\it $d$-rectifiable
} if it can be covered up to Hausdorff $d$-measure zero by Lipschitz images of $\R^{d}$. While classifying rectifiable sets is a classical problem dating back to Besicovitch, starting in the late 80's, geometric measure theorists and harmonic analysts began to study rectifiability in a quantitative manner with an eye on applications to harmonic analysis, particularly singular integrals, analytic capacity, and harmonic measure. 

Much of this work has focused on classifying when Ahlfors regular sets are uniformly rectifiable, which was initiated by David and Semmes in their seminal texts  \cite{DS,of-and-on}. Recall that a set $E\subseteq \bR^{n}$ is {\it Ahlfors $d$-regular} if there is $A>0$ so that
\begin{equation}\label{e:AR}
  r^{d}/A\leq \cH^{d}(B(\xi,r)\cap E) \leq Ar^{d} \mbox{ for }\xi\in E, r\in (0,\diam E)
\end{equation}
and is {\it uniformly rectifiable (UR)} if it has {\it $E$ has big pieces of Lipschitz images (BPLI)}, i.e. there are constants $L,c>0$ so  for all $\xi\in E$ and $r\in (0,\diam E)$, there is an $L$-Lipschitz map $f:\bR^{d}\rightarrow \bR^{n}$ such that $\cH^{d}(f(\bR^{d})\cap B(\xi,r))\geq cr^{d}$.

 Characterizations of uniformly rectifiable sets laid out in the aforementioned texts and in later papers have been indispensable for several important problems in harmonic analysis \cite{NTV14} and harmonic measure  \cite{HM14,HLMN17,AHMNT17}. On one hand, being uniformly rectifiable may imply some nice estimates on a set's multiscale geometry that can be useful for a particular problem (such as \cite{AHMNT17}); conversely, if one is trying to establish that a set has some (uniformly) rectifiable structure in a given problem, the settings of that problem may more easily imply the criteria for one characterization of uniform rectifiability than another (as in \cite{HM14,NTV14,HLMN17}).

We will define some of these criteria here. Let $\cD$ denote the Christ-David cubes for $E$ (see Theorem \ref{t:Christ} below for their definition and the relevant notation we will use below). We say a family of cubes $\cC$ satisfies a {\it Carleson packing condition} if there is a constant $C$ so that for all $R\in \cD$,
\[
\sum_{Q\in \cC \atop Q\subseteq R} \ell(Q)^{d}\leq C\ell(R)^{d}. \]

By Theorem \ref{t:Christ}, for each cube $Q\in \cD$, there is a ball $B_{Q}$ centered on and containing $Q$ of comparable size. Given  two closed sets $E$ and $F$, and $B$ a set we denote
\[
d_{B}(E,F)=\frac{2}{\diam B}\max\left\{\sup_{y\in E\cap B}\dist(y,F), \sup_{y\in F\cap  B}\dist(y,E)\right\}\]

For $C_{0}>0$, and $\ve>0$, let
\[
{\rm BWGL}(C_0,\ve)=\{Q\in \cD| \;\; d_{C_0B_{Q}}(E,P)\geq \ve \;\; \mbox{ for all $d$-planes}P\}. 
\]
BWGL stands for the {\it bilateral weak geometric lemma.} David and Semmes showed in \cite{of-and-on} that $E$ is UR if and only if for every $C_0\geq 1$ there is  $\ve>0$ sufficiently small so that ${\rm BLWG}(C_0,\ve)$ satisfies a Carleson packing condition (with constant depending on $\ve$). 

Another important classification from \cite{of-and-on} is {\it bilateral approximation uniformly by planes} (BAUP): for $R\in \cD$ and $\ve>0$, let 
\begin{align}\label{e:BAUP-def}
   {\rm BAUP}(C_0,\ve)=\{Q\in \cD|\; d_{C_0B_{Q}}(E,U)\geq \ve, \;U\mbox{ is a union of $d$-planes}\}.  
\end{align}
David and Semmes showed that $E$ is UR if and only if ${\rm BAUP}(C_0,\ve)$ satisfies a Carleson packing condition each $C_0>1$ and $\ve>0$ small enough (depending on $C_0$). This was a key tool in \cite{HLMN17} in showing that the weak-$A_{\infty}$ condition for harmonic measure implies UR, and also was key in Nazarov, Tolsa, and Volberg's solution to David and Semmes' conjecture in codimension 1 \cite{NTV14}. 

The focus on Ahlfors regular sets is due to the fact that Hausdorff measure on the set is rather well behaved, and so techniques like stopping-time arguments on dyadic cubes in the Euclidean setting often translate over to this non-smooth setting. The motivation of the current paper, however, is to try and obtain similar estimates on multiscale geometry that exist for uniformly rectifiable sets, but instead for sets that are not Ahlfors regular. Not all quantitative results on rectifiability are in the Ahlfors regular setting. The classical example is the {\it Analyst's Traveling Salesman Theorem} stated below, which will serve as a model for the kind of results we are after.

For  sets $E,B\subseteq \bR^{n}$,  define
\begin{equation}
\beta_{E,\infty}^{d}(B)=\frac2{\diam(B)}\inf\limits_L\sup\{\dist(y,L)|y\in E\cap B\}
\label{e:euclidean-beta}
\end{equation}
where $L$ ranges over  $d$-planes in $\bR^{n}$.
Thus, $\beta_{E,\infty}^{d}(B)\diam(B)$ is the width of the smallest tube containing $E\cap B$.

\begin{theorem}(Jones: $\bR^2$  \cite{Jon90}; Okikiolu: $\bR^n$  \cite{Oki92},Schul \cite{Sch07}) Let  $n\geq 2$.
There is a $C=C(n)$ such that the following holds.
Let $E\subset \bR^n$.
Then there is  a connected set $\Gamma\supseteq E$ such that
\begin{equation}
\cH^{1}(\Gamma)\lec_{n} \diam E+\sum_{Q\in \Delta\atop Q\cap E\neq\emptyset} \beta_{E,\infty}^{1}(3Q)^{2}\diam(Q).
\label{e:betaE}
\end{equation}
Conversely, if $\Gamma$ is connected and $\cH^{1}(\Gamma)<\infty$, then 
\begin{equation}
 \diam \Gamma+\sum_{Q\in \Delta\atop Q\cap \Gamma\neq\emptyset} \beta_{\Gamma,\infty}^{1}(3Q)^{2}\diam(Q) \lec_{n}  \cH^{1}(\Gamma).
\label{e:beta_gamma}
\end{equation}
\label{t:TST}
\end{theorem}
This was first shown by Jones in \cite{Jon90} in the plane, then Okikiolu in $\R^{n}$ \cite{Oki92}, and finally in Hilbert space by Schul \cite{Sch07}, though the statement is different than above. There are also some partial and complete generalizations that hold for curves in other metric spaces \cite{DS17,DS19,FFP07,LS16a,LS16b,Li19}.

An analogue for $d$-dimensional of the second half of Theorem \ref{t:TST} is false due to Fang (see \cite{AS18} for a proof). David and Semmes, however, coined a different definition of a $\beta$-number in terms of which they gave a classification of uniformly rectifiable sets. In \cite{AS18}, the first author and Schul altered their definition to get a version of Theorem \ref{t:TST} for higher dimensional sets, which we describe now.

For a set $E$, a ball $B$, and a $d$-dimensional plane $L$, define

\[
\beta_{E}^{d,p}(B,L)= \ps{\frac{1}{r_{B}^{d}}\int_{0}^{1}\cH^{d}_{\infty}(\{x\in B\cap E| \dist(x,L)>t r_{B}\})t^{p-1}dt}^{\frac{1}{p}}\]
where $r_{B}$ is the radius of $B$, and set 
\[
\beta_{E}^{d,p}(B)=\inf\{ \beta_{E}^{d,p}(B,L)| L\mbox{ is a $d$-dimensional plane in $\bR^{n}$}\}.\]

If $E$ is Ahlfors $d$-regular and we replace $\cH^{d}_{\infty}$ with $\cH^{d}$, this is the $\beta$-number David and Semmes used. However, the $d$-dimensional traveling salesman will be stated for lower content regular sets.

\begin{definition}
A set $E\subseteq \bR^{n}$ is said to be {\it $(c,d)$-lower content regular} in a ball $B$ if 
\[
\cH^{d}_{\infty}(E\cap B(x,r))\geq cr^{d} \;\; \mbox{for all $x\in E\cap B$ and $r\in (0,r_{B})$}.\]
\end{definition}

We can now state the result from \cite{AS18}. It is phrased slightly differently from there, but we justify the reformulation in the appendix. 

\begin{theorem}
Let $1\leq d<n$ and $E\subseteq \bR^{n}$ be a closed set. Suppose that $E$ is $(c,d)$-lower content regular and let $\cD$ denote the Christ-David cubes for $E$. Let $C_0>1$. Then there is $\ve>0$ small enough so that the following holds. 
Let $1\leq p<p(d)$ where
\begin{equation}
\label{e:pd}
p(d):= \left\{ \begin{array}{cl} \frac{2d}{d-2} & \mbox{if } d>2 \\
 \infty & \mbox{if } d\leq 2\end{array}\right. . 
 \end{equation}
For $R\in \cD$, let
    \[
  {\rm BWGL}(R)=  {\rm BWGL}(R,\ve,C_0)=\sum_{ Q\in {\rm BWGL}(\ve,C_0) \atop Q\subseteq R } \ell(Q)^{d}. 
    \]
    and
    \[
    \beta_{E,A,p}(R) :=\ell(R)^{d}+\sum_{Q\subseteq R} \beta_{E}^{d,p}(AB_Q)^{2}\ell(Q)^{d}.
    \]
    Then for $R\in \cD$,
\begin{equation}
\label{e:tst}
\cH^{d}(R)+  {\rm BWGL}(R,\ve,C_0)
 \sim_{A,n,c,p,C_0 \ve} \beta_{E,A,p}(R).
 \end{equation}
\label{t:AS}
\end{theorem}

Since all these values are comparable for all admissible values of $A$ and $p$, below we will simply let 
\[
\beta_{E}(R) :=\beta_{E,3,2}.
\]

The presence of $  {\rm BWGL}(R,\ve,C_0)$ may seem odd, but it disappears in some natural situations. It is just zero if $E$ is $\ve$-Reifenberg flat, for example (c.f. \cite{DT12} for this definition). When $d=n-1$ and $E$ is satisfies Condition B, we have that for any cube $R\subseteq E$,
\[
  {\rm BWGL}(R,\ve,C_0) \lec \cH^{d}(R).\] 
In an upcoming paper, the second author will show that this same estimate occurs for the higher codimensional generalized Semmes surfaces introduced by David in \cite{Dav88} (check there for these definitions). In these scenarios, we then have the more natural looking estimate (more closely resembling \eqref{e:betaE})
\[
 \beta_{E}(R) \sim_{A,n,c, \ve} \cH^{d}(R).
 \]
 
Even in the general case, the higher dimensional Traveling Salesman Theorem above says that ${\rm BWGL}(R,\ve,C_0)$  has some meaning if we compute the sum for a non-Ahlfors regular set: even though it does not necessarily satisfy a Carleson packing condition, it is comparable to the square sum of $\beta$'s for any lower regular set. This opens the question of whether the same holds for sums over other cube families for which a Carleson packing condition would characterize UR sets in the Ahlfors regular setting. 

In this paper, we show this is indeed the case for a large class of the original UR characterizations developed by David and Semmes. A consequence of our results is the following (see Section \ref{s:BAUP} for its proof):

\begin{theorem}
\label{t:baup}
Let $E\subseteq \R^{n}$ be a $(c,d)$-lower content regular set with Christ-David cubes $\cD$. For $R\in \cD$, define
\[
{\rm BAUP}(R,C_0,\ve)=\sum_{Q\subseteq R\atop Q\in{\rm BAUP}(C_0,\ve)}\ell(Q)^{d}.
\]
For all $R\in \cD$, $C_0>1$, and $\ve>0$ small enough depending on $C_0$ and $c$,
\begin{equation}
\cH^{d}(R)+{\rm BAUP}(R,C_0,\ve)\sim_{C_{0},\ve,c} \beta_{E}(R).
\end{equation}
\end{theorem}

We mention one other geometric criteria studied by David and Semmes which we consider: The Local Symmetry' (LS) property is defined as follows. Given $\ve>0$,  let ${\rm LS}(R,\ve,\alpha)$ be the sum of $\ell(Q)^{d}$ over those cubes in $R$ for which  there are $y,z\in B_Q\cap E$ so that $\dist(2y-z,E)\geq \ve r$.

\begin{theorem}
\label{t:LS}
Let $E\subseteq \R^{n}$ be a $(c,d)$-lower content regular set and $\cD$ its Christ-David cubes. Then for $\ve>0$ small enough (depending on $c$), and $R\in \cD$,
\begin{equation}
    \label{e:LS-ineq}
\beta_{E}(R) \sim_{c,\ve} \cH^{d}(R)+{\rm LS}(R,\ve).
\end{equation}
\end{theorem}

This may be surprising, since the Local Symmetry condition is dimensionless, that is, the integer $d$ does not appear in the definition at all, and in fact it could be that, in the ``good" cubes not featured in the sum, $E$ could be very not flat and quite close in the Hausdorff distance to a $(d+1)$-dimensional cube, say, whereas the $\beta$-numbers measure the distance to a $d$-dimensional plane and would be large for these cubes. However, with the assumption that $\cH^{d}(R)$ is finite, this prevents there being too many cubes where $E$ is symmetric but looks like a $(d+1)$-dimensional surface (and this is natural considering that the proof in \cite{DS} connecting LS to flatness of the set uses the Ahlfors regularity of the sets they consider). 

Our method for extending these results is quite flexible: the other characterizations of UR for which we prove analogous statements like those are the Local Convexity (LCV) and Generalized Weak Exterior Convexity (GWEC) conditions, although one could also consider other suitable characterizations in \cite{of-and-on} as well. In fact, our main result is a general test for when a geometric criteria that guarantees uniform rectifiability (like BAUP or BWGL) also implies a result of the form Theorem \ref{t:baup}. Its statement is a bit lengthy to give here, so we postpone it to Section \ref{s:QP}. Loosely speaking, by a geometric criteria $\sP$, we mean a way of splitting up the surface cubes of a set $E$ into ``good" and ``bad" cubes, the good cubes being those cubes near which $E$ satisfies some condition that is trivially satisfied for a $d$-dimensional plane, like being close in the Hausdorff distance to a plane or union of planes. We say it guarantees UR if, whenever we have an Ahlfors regular set, a Carleson packing condition on the bad cubes implies UR. Our result, Lemma \ref{l:meta} below, states that if we have a geometric criterion that guarantees UR and it is, in some sense, continuous in the Hausdorff metric, then a result like Theorem \ref{t:baup} hold with BAUP replaced by $\sP$.


The main lemma that we use may be of independent interest, and has a few forthcoming applications to other problems (see \cite{Azz19,Vil19}) . For the reader familiar with uniform rectifiability, this result says that we can perform a Coronization of lower regular sets by Ahlfors regular sets in a way similar to how David and Semmes construct Coronizations of uniformly rectifiable sets by Lipschitz graphs (see \cite[Chapter 2]{DS}). For the definitions of Christ-David cubes and stopping-time regions, see Section \ref{s:Christ Cubes}.

\begin{mainlemma}
Let $k_0>0$, $\tau>0$, $d>0$ and $E$ be a set that is $(c,d)$-lower content regular. Let $\cD_k$ denote the Christ-David cubes on $E$ of scale $k$ and $\cD=\bigcup_{k\in\bZ} \cD_{k}$. Let $Q_{0}\in \cD_{0}$ and $\cD(k_0)=\bigcup_{k=0}^{k_0}\{Q\in \cD_{k}|Q\subseteq Q_0\}$. Then we may partition $\cD(k_{0})$ into stopping-time regions $Tree(R)$ for $R$ from some collection $Top(k_{0})\subseteq \cD(k_{0})$ with the following properties: 
\begin{enumerate}
\item We have 
\begin{equation}
\label{e:ADR-packing}
\sum_{R \in Top(k_{0})} \ell(R)^{d} \lec_{c,d} \cH^{d}(Q_0).
\end{equation}

\item Given $R\in \Top(k_{0})$ and a stopping-time region $\cT\subseteq \Tree(R)$ with maximal cube $T$, let  $\cF$ denote the minimal cubes of $\cT$ and 
\begin{align}\label{e:d_F}
    d_{\cF}(x) = \inf_{Q \in \cF} \ps{ \ell(Q) + \dist(x,Q)}
\end{align}
For $C_{0}>4$ and $\tau>0$, there is a collection  $\cC$ of disjoint dyadic cubes covering $C_{0}B_{T}\cap E$ so that 
if 
\[
E(\cT)=\bigcup_{I\in \cC} \d_{d} I,\]
where $\d_{d}I$ denotes the $d$-dimensional skeleton of $I$, then the following hold:
\begin{enumerate}[(a)]
\item $E(\cT)$ is Ahlfors regular with constants depending on $C_{0},\tau,d,$ and $c$.
\item We have the containment
\begin{equation}
\label{e:contains}
C_{0}B_{T}\cap E \subseteq \bigcup_{I\in \cC} I\subseteq 2C_{0}B_{T}.
\end{equation}

%
\item $E$ is close to $E(\cT)$ in $C_{0}B_{T}$ in the sense that
\begin{equation}
\label{e:adr-corona}
\dist(x,E(\cT))\lec  \tau d_{\cF}(x) \;\; \mbox{ for all }x\in E\cap C_{0}B_{T}.
\end{equation}
\item The cubes in $\cC$ satisfy
\begin{equation}
\label{e:whitney-like}
\ell(I)\sim \tau \inf_{x\in I} d_{\cF}(x) \mbox{ for all }I\in \cC.
\end{equation}
\end{enumerate}
\end{enumerate}
\end{mainlemma}

The last inequality says that the cubes in $\cC$ are distributed in a sort of Whitney fashion. In particular, if two cubes in $\cC$ are adjacent, then they have comparable sizes. 

Observe that the constants don't depend on $k_0$. The presence of $k_0$ is an artifact of the proof, but in applications we will take $k_0\rightarrow\infty$.

\subsection{Outline}

In Section 3, we prove the Main Lemma and show that a general lower regular set can be approximated by Alhfors regular sets. In Section 4, we show how, if the sum of cubes where a geometric criteria like the BAUP is finite, then we can actually make these Ahlfors regular sets uniformly rectifiable. Using a result of David and Semmes, we know that the sum of $\beta$'s will be finite for these sets, and then that will imply the $\beta$'s for the original set are summable by approximation. After that, we apply our works to get results similar to the Traveling Salesman, but with BWGL replaced by other geometric criteria. In Section 5, we show the same result holds with BWGL replaced by the Local Symmetry and Local Convexity conditions. In Section 6, we consider the BAUP condition and prove Theorem \ref{t:baup}, and in Section 7, we study the GWEC.

\subsection{Acknowledgements}

We'd like to thank Raanan Schul for his useful conversations and encouragement and Matthew Hyde for carefully proofreading the manuscript.

\section{Preliminaries}

\subsection{Notation}

We will write $a\lesssim b$ if there is $C>0$ such that $a\leq Cb$ and $a\lesssim_{t} b$ if the constant $C$ depends on the parameter $t$. We also write $a\sim b$ to mean $a\lesssim b\lesssim a$ and define $a\sim_{t}b$ similarly.

For sets $A,B\subset \bR^{n}$, let 
\[\dist(A,B)=\inf\{|x-y| \; |\; x\in A,y\in B\}, \;\; \dist(x,A)=\dist(\{x\},A),\]
and 
\[\diam A=\sup\{|x-y|\;| \; x,y\in A\}.\]

\subsection{Dyadic cubes}
Let $\cI$ denote the dyadic cubes in $\bR^{n}$, and for $k\in \bZ$, let $\cI_{k}$ be those dyadic cubes of sidelength $2^{-k}$. Give a dyadic cube $I_0$, we will write $\cI(I_0)$ to denote the subfamily of dyadic cubes which are contained in $I_0$. Given some $m \in \bZ$, we set
\begin{align*}
    \cI^m := \bigcup_{k=m}^\infty \cI_k,
\end{align*}
that is, $\cI$ is the family of dyadic cubes with side length at least $2^m$. We will also write 
\begin{align*}
    \cI^m(I_0) := \cI^m \cap \cI(I_0). 
\end{align*}
Finally, given a dyadic cube $I$, we denote by $n(I)$ the integer number so that
\begin{align*}
    \ell(I) = 2^{n(I)}.
\end{align*}
For a cube $I \in \cI$, we write $\partial_d I$ to denote the \textit{$d$-dimensional skeleton} of $I$. For a dyadic cube $I$, its $d$-dimensional skeleton is just the union of its $d$-dimensional faces.

\begin{remark} \label{rem:m}
We may also use the notation $\cI_m$ to mean the family of cubes with side length $\ell(I)=2^{-m}$.
\end{remark}
\subsection{Christ-David  Cubes}
\label{s:Christ Cubes}
We recall the following version of ``dyadic cubes" for metric spaces, first introduced by David \cite{Dav88} but generalized in \cite{Chr90} and \cite{HM12}.

 \begin{theorem} \label{theorem:Christ_cubes}
Let $X$ be a doubling metric space. Let $X_{k}$ be a nested sequence of maximal $\rho^{k}$-nets for $X$ where $\rho<1/1000$ and let $c_{0}=1/500$. For each $n\in\bZ$ there is a collection $\cD_{k}$ of ``cubes,'' which are Borel subsets of $X$ such that the following hold.
\begin{enumerate}
\item For every integer $k$, $X=\bigcup_{Q\in \cD_{k}}Q$.
\item If $Q,Q'\in \cD=\bigcup \cD_{k}$ and $Q\cap Q'\neq\emptyset$, then $Q\subseteq Q'$ or $Q'\subseteq Q$.
\item For $Q\in \cD$, let $k(Q)$ be the unique integer so that $Q\in \cD_{k}$ and set $\ell(Q)=5\rho^{k(Q)}$. Then there is $\zeta_{Q}\in X_{k}$ so that
\begin{equation}\label{e:containment}
B_{X}(\zeta_{Q},c_{0}\ell(Q) )\subseteq Q\subseteq B_{X}(\zeta_{Q},\ell(Q))
\end{equation}
and $ X_{k}=\{\zeta_{Q}\; |\; Q\in \cD_{k}\}$.
\end{enumerate}
\label{t:Christ}
\end{theorem}
For a cube $Q \in \cD_k$, we put
\begin{align}\label{e:child}
    \Child(Q) := \ck{ Q' \in \cD_{k+1} \, |\, Q' \subset Q}. 
\end{align}

%

\begin{definition}\label{d:tree}
A collection $\cT\subseteq \cD$ is a {\it stopping-time region} or {\it tree} if the following hold:
\begin{enumerate}
\item There is a cube $Q(\cT)\in \cT$ that contains every cube in $\cT$.
\item If $Q\in \cT$, $R\in \cD$, and $Q\subseteq R\subseteq Q(\cT)$, then $R\in \cT$.
 \item $Q\in \cT$ and there is $Q'\in \Child(Q)\backslash \cT$, then $\Child(Q)\subset \cT^{c}$.
\end{enumerate}
\end{definition}

\section{Proof of the Main Lemma}

Let $E$ and $Q_{0}$ be as in the Main Lemma. Notice that $Q_0$ is also a lower regular set, although it may not be closed, but we will not need that.

We split the proof into a few subsections. 

\subsection{Frostmann's Lemma}

The first step of the proof follows the proof Frostmann's lemma, but with some extra care. 

Let $I_0=[0,1]^{n}$. Without loss of generality, we assume that $Q_0 \subset I_0$ and that $\diam(Q_0) \geq \ell(I_0)$. 

For $k\in\mathbb{Z}$, let 

\[
\cI_{k}(Q_0)=\{I\in \cI_{k}\; |\; I\cap Q_0\neq\emptyset\}, \;\; \cI^{k}(Q_0)=\bigcup _{j = 0}^{k} \cI_{j}, \;\; \cI(Q_0) = \cI^{\infty}(Q_0)
\]
and 
\[
E_{k}=\bigcup_{I\in \cI_{k}(Q_0)}I.
\]

Let $m\in\mathbb{N}$ (we will choose it later). First let $\mu_{m}^{m}= \cH^{n}|_{E_{m}}2^{(n-d)m}$. In this way, 
\[
\mu_{m}^{m}(I)=\ell(I)^{d}  \;\; \mbox{ for all }I\in \cI_{m}(Q_0).
\]

We define a set of cubes $\Bad$ (which depends on $m$) as follows. First, we immediately add $\cI_{m}(Q_0)$ to $\Bad$. Next, for each $I\in  \cI_{m-1}(Q_0)$, if
\[
\mu_{m}^{m}(I)>2\ell(I)^{d},\]
then we add $I$ to $\Bad$ and define
\[
\mu_{m}^{m-1}|_{I}= \ell(I)^{d} \frac{\mu_{m}^{m}|_{I}}{\mu_{m}^{m}(I)}<\frac{1}{2}\mu_{m}^{m}|_{I}
\]
Otherwise, we set 
\[
\mu_{m}^{m-1}|_{I}= \mu_{m}^{m}|_{I}.\]

Inductively, suppose we have defined $\mu_{m}^{k+1}$ for some integer $k<m$. For $I\in \cI_{k}(E)$, if 
\[
\mu_{m}^{k+1}(I)>2 \ell(I)^{d},\]
then place $I\in \Bad$ and set 
\begin{equation}
\label{e:frost-case-1}
\mu_{m}^{k}|_{I}= \ell(I)^{d} \frac{\mu_{m}^{k+1}|_{I}}{\mu_{m}^{k+1}(I)} <\frac{1}{2} \mu_{m}^{k+1}|_{I}.
\end{equation}
Otherwise, we set 
\begin{equation}
\label{e:frost-case-2}
\mu_{m}^{k}|_{I}= \mu_{m}^{k+1}|_{I}.
\end{equation}
Finally, we put $I_0\in \Bad$. 

Given a cube $I \in \cI$, recall that $n(I)$ is the integer such that $\ell(I)=2^{-n(I)}$. Moreover, for $J \in \cI^m(Q_0)$ (so that $\ell(J) \geq 2^{-m}$), we let $b(J)$ be the numbe of cubes from $\Bad$ properly containing $J$. Now, again with $J \in \cI^m(Q_0)$, let $I_0,..., I_{b(J)} \in \Bad$ be all the bad cubes containing $J$, sot hat $I_j \supset I_{j+1}$ (note that this is consistent with how we defined $I_0$ before). With this notation, we see that $b(I_j)=j$ for all $j$ and if a dyadic cube $I \in \Bad$, then $I=I_{b(I)}$. Let now $I \in \Bad$ and $J \in \cI_m(Q_0)$ so that $J \subset I$. We write


\begin{align}
\mu_{m}^{n(I)}(J)
& =\mu_{m}^{n(I_{b(I)})}(J)\stackrel{\eqref{e:frost-case-1}}{<}\frac{1}{2} \mu_{m}^{n(I_{b(I)})+1}(J) 
\stackrel{\eqref{e:frost-case-2}}{=}\frac{1}{2} \mu_{m}^{n(I_{b(I)+1})}(J)
<\cdots 
\label{e:mu<2^b}
\\ \notag 
& \cdots < \frac{1}{2^{b(J)-b(I)}}\mu_{m}^{n(I_{b(J)})}(J)=\frac{1}{2^{b(J)-b(I)}}\mu_{m}^{m}(J)
=\frac{\ell(J)^{d}}{2^{b(J)-b(I)}}
.\end{align}
Finally, observe that since $Q_{0}$ is $(c,d)$-lower content regular, if $J\cap Q_{0}\neq\emptyset$ and $J\in \cI_{m}(Q_{0})$, then 
\begin{equation}
\label{e:J<Q}
\ell(J)^{d}\lec_{c} \cH^{d}_{\infty}(3J\cap Q_{0})\leq \cH^{d}(3J\cap Q).
\end{equation}
and the cubes $\{3J\; |\; J\in \cI_{m}(Q_{0})\}$ have bounded overlap. Thus,

\begin{align}
\label{e:Bad-est}
\sum_{I\in \Bad} \ell(I)^{d}
& =\sum_{I\in \Bad} \mu_{m}^{n(I)}(I)
=\sum_{I\in \Bad}\sum_{J\in \cI_{m}(Q_0)\atop J\subseteq I} \mu_{m}^{n(I)}(J)\\ \notag 
& \stackrel{\eqref{e:mu<2^b}}{<}\sum_{I\in \Bad}\sum_{J\in \cI_{m}(Q_0)\atop J\subseteq I} 2^{-b(J)+b(I)}\ell(J)^{d}\\ \notag 
& = \sum_{J\in \cI_{m}(Q_0)} \ell(J)^{d}\sum_{I\in \Bad \atop I\supseteq J} 2^{-b(J)+b(I)}
\lec \sum_{J\in \cI_{m}(Q_0)} \ell(J)^{d}\\
&\stackrel{\eqref{e:J<Q}}{\lec}_{c}  \sum_{J\in \cI_{m}(Q_0)} \cH^{d}(3J\cap Q_{0})^{d}
\lec \cH^{d}(Q_0). \notag
\end{align}

For $I\in \Bad$, let 
\[
\mu^{I}:=\mu_{m}^{n(I)}|_{I}.\] 
Note that by construction, for each $J\subseteq I$, we have that $\mu^{I}(J)\leq 2\ell(J)^{d}$, and thus this also holds for all dyadic cubes $J$, even when $J\supseteq I$ or $J\cap I=\emptyset$. In particular, since any ball $B(x,r)$ can be covered by boundedly many dyadic cubes of size comparable to $r$, we obtain that
\begin{equation}
\label{e:frost-upper-regular}
\mu^{I}(B(x,r))\lec r^{d} \mbox{ for all $x\in \R^{n}$, $r>0$}.
\end{equation}
Moreover,
\[
\mu^{I}(I)=\ell(I)^{d}. 
\]

\begin{remark}
Ideally what we'd like to do at this stage is, for each $I\in \Bad$, find the maximal bad cubes $I_{j}\in \Bad$ properly contained in $I$  and define a set  like
\[
E_{I}=\bigcup_{j} \d_{d} I_{j}
\]
where $\d_{d} J$ is the $d$-dimensional skeleton of a cube $J$. Then one can use $\mu^{I}$ to show that $E_{I}$ is an Ahlfors regular set. However, the collection $E_{I}$ will not be suitable for the applications we have in mind, since we need that the sizes of the cubes whose skeletons form $E_{I}$ don't vary too wildly (that is, adjacent cubes should have comparable sizes). This is why more work is needed. 
\end{remark}


\subsection{Trees}

For $I\in \Bad$, we will let $\Tree(I)$ be those cubes in $\cI$ contained in $I$ for which the smallest cube from $\Bad$ that they are properly contained in is $I$, and we will let $\Stop(I)$ be those cubes from $\Bad$ in $\Tree(I)$ properly contained in $I$. 

\begin{remark}
\label{r:treescover}
Observe that $\Stop(I)\subseteq \Tree(I)$, and while the collections $\{\Tree(I): I \in \Bad\}$ do not form a disjoint partition of $\cI^{m}$, they do cover $\cI^{m}$, and they only intersect at the top cubes and stopped cubes. 
\end{remark}
%

\begin{lemma}
For $I\in \Bad$ and $J\in\Stop(I)$,
\begin{equation}
\label{e:adr-on-tree}
2^{d-n-1}\ell(J)^{d}\leq \mu^{I}(J)\leq 2\ell(J)^{d}
\end{equation}
\end{lemma}

\begin{proof}
Note that by construction, for  $I\in \Bad$, and because there are $2^{n}$-dyadic cubes $J\subseteq I$ with $\ell(J)=\ell(I)/2$,
\begin{align*}
\mu_{m}^{n(I)+1}(I)
& =\sum_{J\in \cI_{n(I)+1}\atop J\subseteq I} \mu_{m}^{n(I)+1}(J)
\leq \sum_{J\in \cI_{n(I)+1}\atop J\subseteq I} 2\ell(J)^{d}\\
& = \sum_{J\in \cI_{n(I)+1}\atop J\subseteq I} 2^{1-d}\ell(I)^{d}
\leq 2^{n-d+1}\ell(I)^{d}
\end{align*}
and for $J\in \Stop(I)$,
\[
\mu_{m}^{n(I)+1}(J)=\mu_{m}^{n(J)}(J)=\ell(J)^{d}.\]
Thus,
\[
2\ell(J)^{d}\geq \mu^{I}(J)
=\mu_{m}^{n(I)}(J)
=\mu_{m}^{n(I)+1}(J)\frac{\ell(I)^{d}}{\mu_{m}^{n(I)+1}(I)}
\geq 2^{d-n-1}\ell(J)^{d}.
\]

\end{proof}
%
%
%

Let $M>1$, we will choose it later. For $Q\in \cD(k_0)$ and $I\in \cI(Q_0)$, we write $Q\sim I$ if 
\begin{equation}\label{e:QI}
MB_Q\cap I\neq\emptyset \mbox{ and } \rho\ell(I)\leq \ell(Q) < \ell(I)
\end{equation}
where $\rho$ is as in Theorem \ref{t:Christ}.  Observe that for $m$ large enough,
\begin{equation}
\label{e:IinIm}
\{I\in \cI(Q_0): I\sim Q \mbox{ for some }Q\in \cD(k_0)\}\subseteq \cI^{m}(Q_0).
\end{equation}

Indeed, If $Q\in \cD(k_0)$, this means $\ell(I)\geq \rho \ell(Q)\geq  5\rho^{k_{0}+1}>2^{-m}$ for $m$ large enough, and now we just recall  Remark \ref{r:treescover}. 
\def\Next{{\rm Next}}
We now perform the following stopping-time algorithm on the cubes $\cD(k_0)$. For $R\in \cD(k_0)$ contained in $Q_0$, we let $\Stop(R)$ denote the set of maximal cubes in $R$ from $\cD(k_0)$ that are either in $\cD_{k_{0}}$ or have a child $Q$ for which there is $I\in \Badsh$ such that $Q\sim I$. Observe that if $R\in \cD_{k_{0}}$, then $\Stop(R)=\{R\}$. We then let $\Tree(R)$ be those cubes contained in $R$ that are not properly contained in any cube from $\Stop(R)$, so in particular, $\Stop(R)\subseteq \Tree(R)$. Let $\Next(R)$ be the children of cubes in $\Stop(R)$ that are also in $\cD(k_{0})$ (so this could be empty).

Now let $\Top_{0}=\{Q_{0}\}$, and for $R\in \Top_{k}$, we let \[
\Top_{k+1}=\bigcup_{R\in \Top_{k}}\Next(R),\]
that is, $\Top_{k+1}$ are the children of the cubes in $\Stop(R)$ for each $R\in \Top_{k}$. Let 
\[
\Top = \bigcup_{k\geq 0}\Top_{k}.\]

Note that for each $R\in \Top$, if $R^{1}$ is its parent, then $R^{1}\in \Stop(R')$ for some cube $R'$, and so there is $I_{R}\in \Bad$ with $I_{R}\sim R''$ for some sibling $R''\in \Child(R^{1})$. In particular, the map $R\mapsto I_{R}$ maps boundedly many cubes to one cube, and so 
\begin{equation}
\label{e:top-est}
\sum_{R\in \Top}\ell(R)^{d}
\lec_{M}  \sum_{I\in \Bad}\ell(I)^{d} 
\stackrel{\eqref{e:Bad-est}}{\lec}_{N} \cH^{d}(Q_0).
\end{equation}

The collection $\Top$ is our desired collection and $\{\Tree(R)\; |\; R\in \Top\}$ are the desired stopping-time regions for the Main Lemma  and \eqref{e:ADR-packing} now follows from \eqref{e:top-est}. It remains to verify items (2) of the Main Lemma, which will be the focus of the next two subsections. We will first need a lemma about our trees:

\begin{lemma}
\label{l:trees}
Let $R\in \Top$ and 
\[
S(R):=\{ I\in \cI(Q_0)|\mbox{ $Q\sim I$ for some }Q\in \Tree(R)\}.
\]
Then there is $N_0\lec_{n,M} 1$ and $J_{1}(R),...,J_{N_0}(R)\in \Bad$ so that
\[
S(R)\subseteq \Tree(J_{1}(R))\cup \cdots \cup \Tree(J_{N_0}(R)).\]
\end{lemma}

\begin{proof}
Consider the cubes $I_{1},...,I_{N_0}$ in $\cI(Q_0)$ of maximal size so that $I_{j}\sim R)$ (note that $N_0$ here depends only on $n$ and $M$). Then for $m$ large enough, each $I_{j}$ is contained in $\Tree(J_{j})$ for some $J_{j}\in \Bad$ by \eqref{e:IinIm}.

Now let $I\in S(R)$, so by definition there is $Q\in \Tree(R)$ satisfying \eqref{e:QI}, then $I\subseteq I_{j}\subseteq J_{j}$ for some $j$. If $I\not\in \Tree(J_{j})$, then there is $J\in \Stop(J_i)$ so that $I\subsetneq  J\subseteq I_j \subseteq J_{j}$. Since $I_j\sim R$, $\ell(I_j)< \ell(R)$, so we must have $\ell(J)< \ell(R)$. Thus, if $Q'$ is the maximal ancestor of $Q$ with $\ell(Q')<\ell(J)$, then $\ell(Q')<\ell(R)$, and so $Q'\subsetneq R$. Since $Q\in \Tree(R)$, this implies $Q'\in \Tree(R)$. Since $Q\sim I$ and $\rho\ell(J)\leq  \ell(Q')<\ell(J)$ by the maximality of $Q'$, we also have $Q'\sim J$. So the parent $Q''\subseteq R$ of $Q'$ must be in $\Stop(R)$, but this contradicts $Q$ being in $\Tree(R)$. We let $J_{i}(R)=J_{i}$ and this proves the lemma.
\end{proof} 

\subsection{Smoothing}

We follow the ``smoothing" process of David and Semmes (c.f. \cite[Chapter 8]{DS}). Fix $0< \tau<1$.
For a finite family of cubes $\mathscr{F} \subset \cD$, define the following \textit{smoothing} function: for a point $x \in \R^n$, set 
\begin{align}
    d_{\cF}(x) := \inf_{S \in \cF} \ps{ \ell(S) + \dist(x, S)},
\end{align}
and for a dyadic cube $I \in \cI$, 
\begin{align}
    d_{\cF}(I) := \inf_{x \in I}d_{\cF}(x) = \inf_{ S \in \cF}\ps{\ell(S) + \dist(I,S)}.
\end{align}

We define $\cC_\cF$ to be the set of maximal cubes $I\in \cI(Q_0)$ for which 
\begin{equation}
    \label{eq:SS_defdThin}
    \ell(I) < \tau d_{\cF}(I).
    \end{equation}

The following lemmas are quite standard and appear in different forms depending on the scenario in which they are being applied (depending on between which kinds of cubes, dyadic or not, that $d_{\mathscr{F}}$ is computing), see for example \cite[Lemma 8.7]{DS}. We include their proofs below for completeness.

\begin{lemma}
Let $I, I' \in \cI$. Then, 
\begin{align} \label{eq:SS_prop1}
    d_{\cF}(I) \leq  2\ell(I) + \dist(I, I') + 2\ell(I') +   d_{\cF}(I').
\end{align}
\end{lemma}
\begin{proof}
Let $x,y \in I$ and $x', y' \in I'$. Let also $Q \in \cF$; we have
\begin{align}
    d_{\cF} (x) \leq |x-y| + |y-y'|+ |y' - x'| + \dist(x', Q) + \ell(Q),
\end{align}
simply by triangle inequality and the definition of $d_{\cF}$. Clearly, $|y-y'| \leq \dist(I, I')$; moreover, infimising first over all $Q \in \cF$ and then over all $x' \in I$, we obtain \eqref{eq:SS_prop1}.
\end{proof}

\begin{lemma} \label{lemma:SS_prop2}
Let $I \in  \cC_\cF$; then 
\begin{align} \label{eq:SS_prop2}
\frac{\tau}{2} d_{\cF}(I)\leq    \ell(I) < \tau d_{\cF}(I).
\end{align}
\end{lemma}
\begin{proof}

By \eqref{eq:SS_defdThin},  $\ell(I) < \tau d_{\cF}(I)$, and by definition it is a maximal cube satisfying this inequality.  Hence if $\hat I$ is the parent of $I$, there is a point $z \in \hat I$ with $\tau d_{\cF} (z) \leq 2\ell(I)$. The fact that $d_{\cF}$ is $1$-Lipschitz gives the remaining inequality.
\end{proof}

The following lemma says that if two cubes in $\cC_{\cF}$ are close to each other, then they have comparable size.

\begin{lemma} \label{lemma:sideIs}
Let $I, J \in\cC_{\cF}$ and recall that $\cC_\cF$ depends on a parameter $\tau$. Let $0< \eta<1$ be another small parameter. If
\begin{align}
    \eta^{-1} J \cap \eta^{-1} I \neq \emptyset,
\end{align}
for $\tau^{-1} > 2\sqrt{n}/\eta$,
\begin{align}
    \ell(I) \sim \ell(J).
\end{align}
\end{lemma}
\begin{proof}
It suffices to show that for all $y \in \eta^{-1}J$, 
\begin{align}
\tau^{-1} \ell(J) \sim d_{\cF}(y)
\end{align}
Since $d_{\cF}$ is 1-Lipschitz, we see that
\[
|d_{\cF}(J) - d_{\cF}(y)| \leq \eta^{-1} \diam(J)=\frac{\sqrt{n}}{\eta} \ell(J).\] 
Hence if $\tau^{-1} > 2\sqrt{n}/\eta$,
\begin{align}
    d_{\cF}(y) \geq d_{\cF}(J) - \frac{\sqrt{n}}{\eta}\ell(J)\geq \ps{\tau^{-1}-\frac{\sqrt{n}}{\eta}} \ell(J) \geq \frac{1}{2\tau}\ell(J)
\end{align}
On the other hand, again using the fact that $d_{\cF}$ is 1-Lipschitz, we see that
\begin{align}
       d_{\cF}(y) \lesssim (\eta^{-1} +\tau^{-1}) \ell(J) \lec \tau^{-1} \ell(J).
\end{align}
\end{proof}

\subsection{Constructing an Ahlfors regular set with respect to a tree}

Let $R\in \Top$ and $\cT\subseteq \Tree(R)$ be a stopping-time region, let $T$ denote the maximal cube in $\cT$, $\cF$ be the set of minimal cubes of $\cT$ (that is, those cubes in $\cT$ that don't properly contain another cube in $\cT$). 

Observe that since all the cubes we are working with come from $\cD(k_{0})$ and the number of these cubes in $Q_0$ is finite, the infimum $d_{\cF}$ is attained, and so for each $I\in \cI$ there is $Q_I\in \cF$ so that 
 \begin{equation}
 \label{e:qi}
d_{\cF}(I)=\ell(Q_{I})+\dist(Q_{I},I).
\end{equation}

Let $C_{0}>4$ and set
\begin{align}
\hat{T} = \bigcup\{Q\in \cD\; |\; \ell(Q)=\ell(T),\; Q\cap C_{0}B_{T}\neq\emptyset\}, \label{e:hatT}\\
\cC=\{I \in  \cC_\cF\; |\;  I\cap \hat{T}\neq\emptyset\}, \label{e:cC}
\end{align}
and
\begin{align}
    \hat{E}:= \bigcup_{I\in \cC} \partial_d I.
\end{align}

This set $\hat{E}$ will be our desired $E(\cT)$ as in the statement of the Main Lemma (we just write $\hat{E}$ for short).

\begin{lemma}
For $m$ large enough,
\begin{equation}
\label{e:CinI}
\cC\subseteq \cI^{m}.
\end{equation} 
\end{lemma}
\begin{proof}
Note that by \eqref{eq:SS_prop2}, and because $Q_I\in \cD(k_{0})$, for $I\in \cC$,
\[
\ell(I) \geq \frac{\tau}{2} d_{\cF}(I)\geq \frac{\tau}{2} \ell(Q_{I})\geq \frac{5\tau}{2} \rho^{k_0},\]
and for $\tau$ small enough,
\[
\ell(I)<\tau d_{\cF}(I)\leq \tau (\ell(T)+\dist(I,T)) \leq \tau(C_0+1)\ell(T)
< \frac{1}{5}\ell(Q_0)= 1.\]
Thus, \eqref{e:CinI} follows for $m$ large enough from these two inequalities. 
\end{proof}

\begin{remark}
Note that we definitely don't have that $\cC\subseteq \cI^{m}(Q_0)$, since some cubes in $\cC$ are actually disjoint from $Q_0$. This will cause some difficulties later.
\end{remark}

\begin{lemma}
Part (b) of the Main Lemma holds.
\end{lemma}

\begin{proof}
Firstly, as $C_{0}B_{T}\cap E \subseteq \hat{T}$, we immediately have the first containment, so we just need to show the second containment. 

Note that if $I\in \cC$, then $I\cap Q\neq\emptyset$ for some $Q\in\cD$ with $\ell(Q)=\ell(T)$ and $Q\cap C_{0}B_{T}\neq\emptyset$. Thus,
\[
\dist(I,T)
\leq \dist(Q,T)+\diam Q
\leq C_{0}\ell(T)+2\ell(T)<(C_{0}+2)\ell(T).
\]
Thus,
\[
\diam I =\sqrt{n} \ell(I)<\sqrt{n}\tau d_{\cF}(I)
\leq \sqrt{n}\tau (\dist(I,T)+\ell(T))< \sqrt{n}\tau (C_{0}+3)\ell(T)
\]
so for $\tau>0$ small, $\diam I\leq \frac{C_0}{2}\ell(T)$. Thus, $I\subseteq (3C_{0}/2+2)B_{T}\subseteq 2C_{0}B_{T}$, which proves the lemma. 
\end{proof}

\begin{lemma}
Part (c) of the Main Lemma holds.
\end{lemma}

\begin{proof}
Let $x\in E\cap C_{0}B_{T}\subseteq \hat{T}$. By part (b),  there is $I$ so that $x\in I\in \cC\subseteq \cC_{\cF}$. By definition, $\d_{d} I\subseteq \hat{E}$, and so 
\[
\dist(x,\hat{E})\leq \diam I\leq \sqrt{n}\ell(I)<\sqrt{n} \tau d_{\cF}(I)
\leq \sqrt{n} \tau d_{\cF}(x).
\]

\end{proof}

Moreover, \eqref{e:whitney-like} follows from \eqref{eq:SS_prop2}. Thus, to prove the Main Lemma, all that remains to be shown is the following lemma.

\begin{lemma} \label{lemma:SS_smoothedADR}
Part (a) of the Main Lemma holds, that is, the set $\hat{E}$ is Ahlfors $d$-regular.
\end{lemma}
\begin{proof}

Let $x\in \hat{E}$ and $0<r<\diam \hat{E}\leq 2C_0\ell(T)$. We define
\[
\cC(x, r) = \{I \in \cC \, |\, I \cap B(x,r) \neq \emptyset\}.
\]

 We split into three cases, and in each case we prove first the upper estimate for being Ahlfors regular and then the lower estimate. 

\noindent {\bf Case 1:} $2r\leq d_{\cF}(x)$. Since $d_{\cF}$ is Lipschitz, this means $d_{\cF}(y)\geq d_{\cF}(x)-|x-y|$, and so if $I\in \cC(x,r)$, $y\in I$ is so that $d_{\cF}(I)=d_{\cF}(y)$, and $z\in I\cap B(x,r)$, then $|z-y|\leq \diam I=\sqrt{n} \ell(I)$, and so 
\begin{align*}
\ell(I)
& \stackrel{\eqref{eq:SS_prop2}}{\sim} \tau d_{\cF}(I)
=\tau d_{\cF}(y)
\geq \tau (d_{\cF}(x)-|x-y|)\\
& \geq \tau(2r-|x-z|-|z-y|)
 \gec \tau(r-\sqrt{n}\ell(I))
\end{align*}
and so for $\tau\ll \sqrt{n}$ we have $\ell(I)\gec \tau r$. This implies $\# \cC(x,r)\lec_{n,\tau} 1$, and so it is not hard to show that 
\[
\cH^{d}(\hat{E}\cap B(x,r))\sim_{n,\tau} 1.
\]

\noindent {\bf Case 2:} $8\ell(T)>2r> d_{\cF}(x)$.

Before we proceed, we record a few estimates. First, for $I\in \cC(x,r)$, if $2r> d_{\cF}(x)$, 
\begin{equation}
\label{e:ismall}
\tau^{-1} \ell(I) \stackrel{\eqref{eq:SS_prop2}}{<}  d_{\cF}(I)\leq d_{\cF}(y)
\leq d_{\cF}(x)+|x-y| <2r+r=3r
\end{equation}

Next, note that for all $I\in \cC$, $\ell(Q_I) \leq d_{\cF}(I)$. Let $Q_{I}'$ be the largest cube in $\cT$ containing $Q_{I}$ so that $\ell(Q_{I}')\leq d_{\cF}(I)$. 

\begin{lemma}
If $x\in \hat{E}$ and $d_{\cF}(x)< 2r<24 \ell(T)$, then  
\begin{equation}
\ell(Q_{I}')\sim_{\tau} \ell(I).
\label{e:QI'simI}
\end{equation}
\end{lemma}

\begin{proof}
If $Q_{I}'=T$, then $Q_{I}=T$, so
\begin{equation}
\ell(T)\leq d_{\cF}(I) \stackrel{\eqref{e:ismall}}{<}3r \lec \ell(T)
\end{equation}
and so $\ell(Q_{I}')=\ell(T)\sim_{\tau}  \ell(I)$. Otherwise, if $\ell(Q_{I}')<\ell(T)$, then $\ell(Q_{I}')\sim d_{\cF}(I)\stackrel{\eqref{eq:SS_prop2}}{\sim}_{\tau} \ell(I)$ by maximality of $Q_{I}'$ (indeed, if $\ell(Q_{I}')<\rho d_{\cF}(I)$, then its parent $Q_{I}''$ satisfies $\ell(Q_{I}'')<d_{\cF}(I)$ and $Q_{I}''\in \cT$ since $Q_{I}'\subsetneq T$, but this contradicts the maximality of $Q_I'$). This proves the lemma. 
\end{proof}

Recall \eqref{e:CinI} and let 
\[
\cC_{1}(x,r)=\{I\in \cC(x,r): I\cap Q_0 \neq\emptyset\} =\cC(x,r)\cap \cI^{m}(Q_0),\]
\[ \cC_{2}(x,r)=\cC(x,r)\backslash\cC_{1}(x,r).
\]

\begin{lemma}
If $x\in \hat{E}$ and $d_{\cF}(x)< 2r<24 \ell(T)$, then  
\begin{equation}
\label{e:C1}
\sum_{I\in \cC_{1}(x,r)}\ell(I)^{d} \lec r^{d}.
\end{equation}
\end{lemma}

\begin{proof}

We need an estimate like $\ell(I)^{d}\lec \cH^{d}_{\infty}(I\cap Q_0)$, but this may not necessarily be true: of course $I\cap Q_0\neq\emptyset$ since $I\in \cC_{1}(x,r)$, but it could be that $I$ only intersects $Q_0$ at a corner of $I$ so $ \cH^{d}_{\infty}(I\cap Q_0)$ could be very small compared to $\ell(I)^{d}$. To overcome this, we associate to $I$ a neighboring dyadic cube that does intersect $E$ in a large set. Let ${\rm Nei(I)}$ be the set of dyadic cubes $J\subseteq 3I$ with $\ell(J)=\ell(I)$. Then 
\[
\ell(I)^{d}\lec \cH^{d}_{\infty}(3I\cap Q_0)
\leq \sum_{J\in {\rm Nei(I)}} \cH^{d}_{\infty}(J\cap Q_0).
\]
Hence there is $I'\in {\rm Nei(I)}$ so that 
\[
\cH^{d}_{\infty}(I'\cap Q_0)\gec \ell(I)^{d}. 
\]
Since $I'\subseteq 3I$, we know that 
\[
\dist(I',Q_{I}')\leq \diam I+ \dist(I,Q_{I})\leq \sqrt{n}\ell(I)+d_{\cF}(I)
\stackrel{\eqref{eq:SS_prop2}}{\lec} 
 \tau^{-1} \ell(I)\sim \ell(Q_I').
\]

As $\ell(I)\sim_{\tau}\ell(Q_I')$,  for $M\gg \tau^{-1}$ large enough, $MB_{Q_{I}'}\cap I'\neq\emptyset$, and since $Q_{I}'\in \cT\subseteq \Tree(R)$ and $I'\in \cI^{m}(Q_0)$ (because $I'\cap Q_{0}\neq\emptyset$ and $I\in \cI^{m}$ by \eqref{e:CinI}), this implies $I'\sim Q_{I}'$, and so $I'\in S(R)$ (where $S(R)$ is as in Lemma \ref{l:trees}). In particular, there is $J_{i}=J_{i}(R)$ so that $I'\in \Tree(J_{i})$ by  Lemma \ref{l:trees}. We will use this fact shortly, but we need one more estimate: We now claim that
\begin{equation}
\label{e:overlap}
\sum_{I\in \cC_1(x,r)} \one_{I'}\lec \one_{B(x,2r)}.
\end{equation}
Indeed, if $y\in I_{1}'\cap\cdots \cap I_{\ell}'$ for some distinct $I_{1},...,I_{\ell}\in \cC_1(x,r)$, then the $I_{j}$ are disjoint and $y\in 3I_{1}\cap\cdots \cap 3I_{\ell}$, so Lemma \ref{lemma:sideIs} implies they have sizes all comparable to $I_{1}$ and are also contained in $9I_{1}$ (assuming $I_1$ is the largest). Thus if $|A|$ denotes the Lebesgue measure of a set $A$,
\[
\ell |I_{1}|
\sim \sum_{i=1}^{\ell} |I_{i}|=\left| \bigcup_{i=1}^{\ell} I_{i}\right|\leq |9I_{1}|\]
which implies $\ell\lec 1$, thus, $\sum_{I\in \cC_1(x,r)} \one_{I'}\lec 1$. Finally, note that
\[
\diam I=\sqrt{n} \ell(I) \stackrel{\eqref{eq:SS_prop2}}{<} \tau \sqrt{n} d_{\cF}(I)\stackrel{\eqref{e:ismall}}{<}3\sqrt{n}\tau r
\]
and since $I$ and $I'$ touch, $\dist(x,I')\leq \diam I+r<(3\sqrt{n}\tau+1)r$,  so for $\tau>0$ small enough, $I'\subseteq B(x,2r)$.  Thus, \eqref{e:overlap} follows. 

Thus,
\begin{align*}
\cH^{d}(\hat{E}\cap B(x,r))
& \lec \sum_{I\in \cC_1(x,r)}\ell(I)^{d}
  \lec \sum_{I\in \cC_1(x,r)} \cH_{\infty}^{d}(I'\cap Q_0)\\
& \leq \sum_{I\in \cC_1(x,r)} \sum_{i=1}^{N_{0}} \sum_{J\in \Stop(J_{i})\atop J\subseteq I'} (\diam J)^{d}\\
& \stackrel{\eqref{e:adr-on-tree}}{\lec} \sum_{I\in \cC_1(x,r)} \sum_{i=1}^{N_{0}} \sum_{J\in \Stop(J_{i})\atop J\subseteq I'} \mu^{J_{i}}(J) \\
& \leq \sum_{I\in \cC_1(x,r)} \sum_{i=1}^{N_{0}}  \mu^{J_{i}}(I') 
\stackrel{\eqref{e:overlap}}{ \lec} \sum_{i=1}^{N_{0}}  \mu^{J_{i}}(B(x,2r)) \lec r^{d}.
\end{align*}

This proves \eqref{e:C1}. 

\end{proof}

\begin{lemma}
If $x\in \hat{E}$ and $d_{\cF}(x)< 2r<8 \ell(T)$, then  
\begin{equation}
\label{e:C2}
\sum_{I\in \cC_{2}(x,r)}\ell(I)^{d} \lec r^{d}.
\end{equation}
\end{lemma}

\begin{proof}

For $I\in \cC_{2}(x,r)$, let $\tilde{Q}_{I}$ denote the child of $Q_{I}'$ containing the center of $Q_{I}'$. We claim that the cubes  $\{\tilde{Q}_{I}:I\in \cC_{2}(x,r)\}$ have bounded overlap.

Indeed, suppose there were $I_{1},....,I_{\ell}\in \cC_{2}(x,r)$ distinct and a point
\[
y\in \bigcap_{j=1}^{\ell} \tilde{Q}_{I_j}.
\]
We can assume that $\tilde{Q}_{I_{1}}$ is the largest, and since they are all cubes, this implies $\tilde{Q}_{I_1}\supseteq \tilde{Q}_{j}$ for all $j$. Since 
\begin{equation}
\label{e:I_jtoQ1}
\dist(I_j,\tilde{Q}_{I_1})\leq \dist(I_j,\tilde{Q}_{I_j})\leq \dist(I_j,Q_{I_{j}})\leq d_{\cF}(I_{j})
\stackrel{\eqref{eq:SS_prop2}}{\lec} \tau^{-1}\ell(I_{j})
\end{equation}
and the $I_{j}$ are disjoint, and because $\ell(I_{j})\stackrel{\eqref{e:QI'simI}}{\sim}_{\tau} \ell(Q_{I_{j}}')\sim \ell(\tilde{Q}_{I_j})$, for given $\ve>0$, there can be at most boundedly many $I_{j}$ (depending on $\ve$ and $\tau$) for which $\diam I_{j}\geq \ve  \ell(\tilde{Q}_{I_1})$. For the rest of the $j$, we have that 
\[
\dist(I_j,\tilde{Q}_{I_1})
\stackrel{\eqref{e:I_jtoQ1}}{\lec}\tau^{-1} \ell(I_j)< \frac{\ve}{\tau} \ell(\tilde{Q}_{I_1}),
\]
so for $\ve>0$ small enough, and recalling that $\rho<c_{0}/2$ in Theorem \ref{t:Christ}, this implies $I_{j}\subseteq c_0 B_{Q_{I_{j}}'}$. Since $I_{j}\cap \hat{T}\neq\emptyset$ and the balls $\{c_0 B_Q:Q\in \cD_k\}$ are disjoint for each $k$ by Theorem \ref{t:Christ}, this means $\emptyset\neq I_{j}\cap Q_{I_{j}}'\subseteq I_{j}\cap Q_{0}$, and so $I_{j}\in \cC_{1}(x,r)$, which is a contradiction since we assumed $I_j\in \cC_{2}(x,r)$. Thus, there are no other $j$, and so $\ell\lec_{\ve} 1$. This finishes the proof that the sets $\{\tilde{Q}_{I}:I\in \cC_{2}(x,r)\}$ have bounded overlap.\\ 

Fix $I\in \cC_{2}(x,r)$ and let $J\in \cC$ so that $J\cap \frac{c_0}{2}B_{\tilde{Q}_{I}} \neq\emptyset$. Then $\ell(J)<\tau d_{\cF}(J)\leq \tau \ell(\tilde{Q}_{I})$, so for $\tau$ small enough, $J\subseteq c_0 B_{\tilde{Q}_{I}}$. Thus, if 
\[
\{J_{I}^{i}\}_{i=1}^{L_{I}} = \{J\in \cC: J\cap\frac{c_0}{2}B_{\tilde{Q}_{I}}\neq\emptyset\},
\]
since the $\tilde{Q}_{I}$ have bounded overlap, so do the cubes 
\[
\{J_{I}^{i}: i=1,...,L_I, \;\; I\in \cC_{2}(x,r)\}.
\]
For  $I\in \cC(x,r)$ and $i=1,...,L_I$,
\[
\dist(I,J_{I}^{i})\leq \dist(I,\tilde{Q}_{I})\leq \dist(I,Q_I)\leq d_{\cF}(I)<2r,
\]
hence $J_{I}^{i}\in \cC_{1}(x,3r)$. Now we have by our assumptions that
\[
d_{\cF}(x)<2r<2\cdot(3r)=3\cdot (2r) < 3\cdot 8 \ell(T)=24\ell(T).
\]
Thus, \eqref{e:C1} holds for $3r$ in place of $r$, and so
\begin{align*}
\sum_{I\in \cC_{2}(x,r)}\ell(I)^{d}
& \sim \sum_{I\in \cC_{2}(x,r)}\ell(\tilde{Q}_{I})^{d}
\sim_{c}  \sum_{I\in \cC_{2}(x,r)}\cH^{d}_{\infty}\ps{\frac{c_{0}}{2} B_{\tilde{Q}_{I}}} \\
& \lec  \sum_{I\in \cC_{2}(x,r)}\sum_{J\in \cC \atop J\cap\frac{c_0}{2}B_{\tilde{Q}_{I}}\neq\emptyset} \ell(J)^{d}
=\sum_{I\in \cC_{2}(x,r)}\sum_{i=1}^{L_{I}}\ell(J_{I}^{i})^{d} \\
& \lec \sum_{J\in \cC_{1}(x,3r)^{d}}\ell(J)^{d}
\lec r^{d}
\end{align*}
where we used the bounded overlap property in the penultimate inequality.

\end{proof}

Thus, combining the two previous lemmas, we have that 
\[
\cH^{d}(\hat{E}\cap B(x,r))
\lec \sum_{i=1}^{2} \sum_{I\in \cC_{i}(x,r)}\ell(I)^{d}
\lec r^{d}.
\]

Now to complete the proof in this case, we need to show the reverse estimate. Let $I\in \cC(x,r/2)$. Then \eqref{e:ismall} implies that for $\tau$ small enough, $I\subseteq B(x,r)$. Moreover, since $I\in \cC$, $I\cap Q\neq\emptyset$ for some $Q\subseteq \hat{T}$ with $\ell(Q)=\ell(T)$. If $I\in \cC(x,r/2)$ is the cube so that $x\in \d_{d}I$, then for $\tau$ small,
\[
\dist(x,Q)
\leq \diam I \stackrel{\eqref{e:ismall}}{\leq}  3\sqrt{n}\tau r<\frac{r}{4}
\]
Thus, there is $y\in Q\cap B(x,r/4)$, and so we can find a subcube $Q'\subseteq B(x,r/2)\cap Q$ containing $y$ so that $\ell(Q')\sim r$ and the cubes from $\cC(x,r/2)$ cover $Q'$. Thus,
\begin{align*}
\cH^{d}(B(x,r)\cap \hat{E})
& \geq \sum_{I\in \cC(x,r/2)}\cH^{d}(\d_{d}I)
\sim \sum_{I\in \cC(x,r/2)} \ell(I)^{d}\\
& \gec \cH^{d}_{\infty}(Q')\gec \ell(Q')^{d}\sim r^{d}.
\end{align*}

\noindent {\bf Case 3:} $2C_{0}\ell(T)>r>4\ell(T)$. 

Note that by the previous case,  
\[
\cH^{d}(B(x,r)\cap 2B_{T}\cap \hat{E})
\leq \cH^{d}(2B_{T}\cap\hat{E})\lec \ell(T)^{d}\lec r^{d}. 
\]
So to prove upper regularity, we just need to verify
\[
\cH^{d}(B(x,r)\cap \hat{E}\backslash 2B_{T})\lec r^{d}.
\]
If $I\cap B(x,r)\backslash 2B_{T}\neq\emptyset$, and if $y\in I\backslash 2B_{T}$,
\begin{align*}
\ell(I)
& \sim \tau d_{\cF}(I)
\geq \tau( d_{\cF}(y)- \diam I)
\geq \tau\dist(y,T)-\tau \sqrt{n}\ell(I)\\
& \geq \tau \ell(T)-\tau\sqrt{n}\ell(I)
\end{align*}
and so for $\tau$ small enough,
\[
 \frac{\tau}{2}\ell(T)\leq \ell(I)
 \]
 Moreover, since $I\cap B(x,r)\neq\emptyset$, $x\in \hat{E}$, and $T\subseteq \bigcup_{J\in \cC_{\cF}} J$,
 
 \begin{align*}
 \ell(I) 
 &   <\tau d_{\cF}(I)\leq\tau( \ell(T)+\dist(I,T)) \\
 & < \tau (\ell(T)+\diam I + 2r+\dist(x,T))\\
&  <\tau(\ell(T)+\sqrt{n}\ell(I) + 4C_{0} \ell(T)+\diam \hat{E}) \\
& \lec \tau(\ell(T)+\ell(I))
\end{align*}
So for $\tau>0$ small enough, we also have $\ell(I)\lec \tau \ell(T)$, hence $\ell(I)\sim \tau \ell(T)$. There can only be at most boundedly many disjoint cubes $I\in  \cC$ with $\ell(I)\sim \tau \ell(T)$, and so
\[
\cH^{d}(\hat{E}\cap B(x,r)\backslash 2B_{T})\lec \ell(T)^{d}\sim r^{d}. 
\]

For the lower bound, if $x\in \hat{E}\cap 2B_{T}$, then $r>4\ell(T)$ implies by the previous case that
\[
\cH^{d}(\hat{E}\cap B(x,r))\geq \cH^{d}(\hat{E}\cap 2B_{T})\gec \ell(T)^{d}\sim r^{d}.
\]
Alternatively, if $x\in \hat{E}\backslash 2B_{T}$, then by the arguments above, if $I\in \cC$ contains $x$, then $\ell(I)\sim \tau \ell(T)\sim \tau r$, so for $\tau$ small enough, $I\subseteq B(x,r)$. Thus,
\[
\cH^{d}(\hat{E}\cap B(x,r))\geq \cH^{d}(\d_{d}I)\sim \ell(I)^{d}\sim r^{d}. 
\]
This completes the proof.

\end{proof}

This finishes the proof of the Main Lemma.

\section{A general lemma on quantitative properties}
\label{s:QP}
We now want to apply the approximation by Ahlfors regular sets obtained in the previous section to derive quantitative bounds on the sum of the $\beta$ coefficients. The method we present is quite easy and general. The idea is the following: let us pick one of the quantitative properties described by David and Semmes. For example, the BAUP (which stands for bilateral approximation by union of planes) (see \cite{of-and-on}, II, Chapter 3), the GWEC (generalised weak exterior convexity) (see \cite{of-and-on}, II, Chapter 3), or the LS (local symmetry),  see \cite{DS}, Definition 4.2. On each cube $R \in \Top$, we run a stopping time on $\Tree(R)$ where we stop whenever we meet a cube which does not satisfy the chosen property. By doing so, we obtain a new tree and consequently a new approximating Ahlfors regular set. This time, however, this set will turn out to be uniformly rectifiable exactly because it approximates $E$ at those scales where $E$ is very well behaved. 

Let us try to make all this precise. 
\begin{definition}[Quantitative property]
 By a quantitative property (QP) $\sP$ of $E$ we mean a finite set of real numbers $\{p_1,...,p_N\}$ with $0<p_1 \leq 1$ together with two subsets of $E \times \R_+ = E\times (0,\infty)$
 \begin{align*}
     \cG^\sP=\cG^{\sP}(p_1,...,p_N) \mbox{ and }  \cB^\sP= \cB^\sP(p_1,...,p_N), 
 \end{align*} 
 which depend on $\{p_1,...,p_N\}$, such that 
 \begin{align}
     \cG^\sP \cup \cB^\sP = E \times \R_+ \mbox{ and } \cG^\sP \cap  \cB^\sP = \emptyset.
 \end{align}
 We will call $\{p_1,...,p_N\}$ the {\it parameters} of $\sP$.
\end{definition}
If we want to specify the subset $E$ upon which we are applying a quantitative property $\sP$, we may write, for example, $\cG_E^\sP$, or $\cB_E^\sP$.
Let us give a few examples of quantitative properties described in the book \cite{of-and-on}:
\begin{itemize}
    \item[{\bf BWGL:}] The so-called `Bilateral Weak Geometric Lemma' (BWGL) is a quantitative property. Given a real number $\epsilon>0$, for each pair $(x,r) \in E \times \R_+$, BWGL asks whether there exists a plane $P$ so that
    \begin{align*}
        d_{B(x,r)}(E, P) < \epsilon.
    \end{align*}
    If one such a plane exists, then we put $(x,r) \in \good{BWGL}$; if not, then $(x,r) \in \bad{BWGL}$. This is clearly a partition of $E \times \R_+$. Hence BWGL is a QP with parameter $\ve$.
    \item[{\bf LS:}] The `Local Symmetry' (LS) property is defined as follows. Given $\ve>0$, for each pair $(x,r)\in E\times \R_{+}$, we say $(x,r)\in \bad{LS}(\ve,\alpha)$ if there are $y,z\in B(x,r)\cap E$ so that $\dist(2y-z,E)\geq \ve r$.
    \item[{\bf LCV}] For the quantitative property `Local Convexity' (LCV), we define $\bad{LCV}$ to be those $(x,r)\in E\times \R_{+}$ for which there are $y,z\in B(x,r)\cap E$ such that $\dist((y+z)/2,E)\geq \ve r$. 
%
    \item[{\bf WCD:}] Let  two positive numbers $C_0$ and $\ve$ be given. The `Weak Constant Density' (WCD) condition asks the following: for $(x,r) \in E \times \R_+$, does a measure $\mu_{x,r}$ exists, such that
    \begin{align*}
        & \spt(\mu_{x,r}) = E; \\
        & \mu_{x,r} \mbox{ is Ahlfors  } d-\mbox{regular with constant } C_0 \geq 1;\\
        & |\mu_{x,r}(y,s) -s^d | \leq \epsilon t^d \mbox{ for all } y \in E \cap B(x,r) \mbox{ and } 0 < s \leq r.
    \end{align*}
    If one such a measure $\mu_{x,r}$ exists, then we put $(x,r) \in \good{WCD}(C_0^{-1}, \epsilon)$. If not, then $(x,r) \in \bad{WCD}(C_0^{-1},\ve)$. This is clearly a partition of $E \times \R_+$ and so WCD is a QP with parameters $(C_0^{-1}, \ve)$. 
\item[{\bf BP:}] Let us give one more example. Let $1 \geq \theta>0$ be a positive real number. The `Big Projection' (BP) condition asks if for a pair $(x,r)$, there exists a $d$-dimensional plane $P$ such that
    \begin{align*}
        | \Pi_P(B(x,r) \cap E) | \geq \theta r^d,
    \end{align*}
    where $\Pi_P$ is the standard orthogonal projection onto $P$ and $| \cdot |$ is the $d$-dimensional Lebesgue measure on $P$.  We put $(x,r) \in \good{BP}(\theta)$ if this is the case; otherwise $(x,r) \in \bad{BP}(\theta)$. Thus BP is a QP with parameter $\theta>0$.
\end{itemize}

\begin{definition}\label{d:haus-cont} Fix a (small) parameter $\ve_1>0$ and two (large) constants $C_1, C_2 \geq 1$ and
let $\sP$ be a quantitative property with parameters $\{p_1,...,p_N\}$. We say that $\sP$ is $(\epsilon_1, C_1, C_2)$-\textit{continuous}, if there exist positive constants $0<c_1,...,c_N<\infty$ depending on $\epsilon_1$ and $C_{1}$ such that the following holds. Let $E_1$ and $E_2$ be two subsets of $\R^n$ and let $B=B(x_B, r_B)$ be a ball so that 
\begin{align*}
    & B \mbox{ is centered on } E_1; \\
    & (x_B, r_B) \in \cG^\sP_{E_1}(p_1,...,p_N);\\ 
    & d_{C_2 B}(E_1,E_2) < \epsilon.
\end{align*}
If $B'=B(x_{B'}, r_{B'})$ is a ball so that
\begin{align*}
    & B' \mbox{ is centered on } E_2;\\
    & C_2 B' \subset B;\\ 
    & r_{B'} \geq \frac{r_B}{C_1},
    \end{align*}
then 
\begin{align} \label{e:sP-h-c}
    (x_{B'},r_{B'}) \in \cG^{\sP}_{E_2}(c_1p_1,...,c_Np_N).
\end{align} 

\end{definition}

\begin{remark}\label{rem:stability}
In particular a continuous quantitative property is \textit{monotonic} (or \textit{stable}) in the following sense; take a set $E$ and a ball $B$ centered on $E$ with $(x_B, r_B) \in \cG^\sP_E(p_1,...,p_N)$. If we assume that $\sP$ is continuous and we take $E_1=E_2=E$ in Definition \ref{d:haus-cont}, then we see that $(x_{B'}, r_{B'}) \in \cG^\sP_E(c_1p_1,...,c_Np_N)$ whenever $C_2 B' \subset B $ and $r_{B'} \geq \frac{r_B}{C_1}$. 
\end{remark}
Let us look at our concrete examples of QP, and see whether they are continuous, and thus stable. 
\begin{itemize}
    \item One can quite easily check that BWGL, LS, LCV, and BP are stable quantitative properties. 
    \item On the other hand, the WCD is not.
\end{itemize}
%

\begin{definition}[QP guaranteeing uniform rectifiability] We say a QP (with parameters $p_{1},...,p_{N}$) {\it guarantees uniform rectifiability}  for Ahlfors $d$-regular sets with constant $C_{1}$ if, whenever $A$ is Ahlfors $d$-regular with constant $C_{1}$ and 
\begin{align}\label{e:QPUR-carleson}
    \chara_{\cB^{\sP}(p_1,...,p_N)}\,\dr d\hd|_A \mbox{ is a Carleson measure on } A \times R_+,
\end{align}
then $A$ is a uniformly rectifiable set. 
Conversely, if $A$ is uniformly rectifiable, then we say a QP (with parameters $p_{1},...,p_{N}$) is {\it guaranteed by uniform rectifiability} if the measure in \eqref{e:QPUR-carleson} is a Carleson measure for the  parameters $(p_1,...,p_N)$.
\end{definition}

Let us go back to our examples. 
\begin{itemize}
    \item In the two monographs \cite{DS} and \cite{of-and-on}, David and Semmes prove that the properties BWGL (\cite{of-and-on}, II.2, Proposition 2.2), 
    and WCD are indeed examples of QP guranteeing uniform rectifiability (see \cite{of-and-on}, I.2, Proposition 2.56, and \cite{tolsa}, Theorem 1.1). To further comment on the remark above, consider BWGL: if an Ahlfors $d$-regular set $A$ is uniformly rectifiable, then there exists a universal constant $\epsilon_0>0$ so that for all $0< \epsilon < \epsilon_0$, we have that
    \begin{align}\label{e:BWGL-carleson}
    \int_{B \cap E} \int_0^R \chara_{\bad{BWGL}(\epsilon)}
    (x,r)\, \dr\, d \hd(x) \leq C(\epsilon) r_B^d,
\end{align}     
for all balls $B$ centered on $E$ with $r_B \leq \diam(E)$. In general, one may have that $C(\epsilon) \to \infty$ as $\epsilon \to 0$. On the other hand, it suffices to find a sufficiently small $\epsilon>0$ for which \eqref{e:BWGL-carleson} holds to prove that $A$ is uniformly rectifiable. 
    \item The property BP, on the other hand, does not guarantee uniform rectifiability. The standard 4-corner Cantor set is purely unrectifiable but still satisfy the Carleson measure condition above since it has large projections in some directions (although of course not {\it many} directions), see \cite[Part III Chapter 5]{Dav91}.
\end{itemize}

Let now $\sP$ be a continuous quantitative property with parameters $\{p_1,...,p_N\}$. For a cube $Q_0 \in \cD$,  we let 
\begin{align}\label{e:cube-B-G}
& \cB^\sP(Q_0) = \cB^\sP(Q_0,p_1,...,p_N)\\
 & \enskip \enskip \enskip \enskip := \Big\{ Q \in \cD \, |  Q \subset Q_0; (\zeta_Q, \ell(Q))  \in \cB^\sP \Big\};\nonumber \\
& \cG_{}^\sP(Q_0) = \cG^\sP(Q_0,p_1,...,p_N) := \cD(Q_0) \setminus \cB_{}^\sP.\nonumber 
\end{align}
Thus we put
\begin{align*}
   \sP(Q_0,p_1,...,p_N) := \sP(Q_0) := \sum_{ Q \in \cB_{Q_0}^\sP } \ell(Q)^d. 
\end{align*}

The following is the main result of this section. In later sections, we will show how the comparability results (Theorems \ref{t:baup} and \ref{t:LS}) follow as corollaries. 

\begin{lemma}\label{l:meta}
Let $E \subset \R^n$ be a $(c,d)$-lower content regular set, and let $0<\ve<1$, $C_{2}\geq 1$, and  $C_1>4C_2/\rho$. There is $C_{0}'$ depending on $c$ so that the following holds. Let $\sP$ be a QP of $E$ with parameters $\{p_1,...,p_N\}$ such that
\begin{align}
  & \sP \mbox{ is } (\ve, C_1,C_2)\mbox{-continuous. }  \mbox{ with constants $c_{1},...,c_{N}$} \label{e:sPmeta-cont}\\
    & \sP \mbox{ guarantees (and is guaranteed by) UR for $C_{0}'$-Ahlfors $d$-regular sets} \\
    & \qquad \mbox{for parameters $c_1p_1,...,c_Np_N$};\label{e:sPmeta1} \notag \\
\end{align}
Then for any $Q_{0}\in \cD$ \begin{align} \label{e:meta-beta-est}
    \beta_{E}(Q_0) \lesssim_{c,C_1, \ve} \hd(Q_0) + \sP(Q_0, c_1p_1,...,c_Np_N).
\end{align}
\end{lemma}

The proof of Lemma \ref{l:meta} will take up the rest of this section.
Let us get started by first modifying the tree structure of $\Top(k_0)$, as in the statement of the Main Lemma by introducing a further stopping condition which is related to the QP $\sP$. Let $R\in \Top(k_0)$ and $R'\in \Tree(R)$. Let $\wt\Stop(R')$ be the maximal cubes in $\Tree(R)$ that are either in $\Stop(R)$ or contain a child in $\cB^{\sP}(Q_0)$, and let $\wt \Tree(R')$ be the subfamily of cubes $Q \in \Tree(R)$ contained in $R$ that are not properly contained in a cube from $\wt\Stop(R')$.  
%
In other words, $\wt \Tree(R')$ is a pruned version of $\Tree(R)$, where we cut whenever we found a cube $Q \in \cB^\sP_{\cD}$. 

Let $\Next_{0}(R)=\{R\}$ and for $j\geq 0$, if we have defined $\Next_{k}(R) $, let 
\[
\Next_{j+1}(R) = \bigcup_{R'\in \Next_{j}}\bigcup_{Q\in \wt\Stop(R')} \Child(Q).
\]

%
This process terminates at some integer $K_R$ since $\Tree(R)$ is finite. Enumerate $\Next_{j} = \{Q_{i}^{j}\}_{i=1}^{i_{j}}$.

\begin{lemma}\label{l:Tree-bound}
Let $R \in \Top(k_0)$ and let $0 \leq j \leq K_R$ and $1\leq i\leq i_j$. Then there exists a constant $c_1<1$ so that
\begin{align} \label{e:beta-bound-tree}
    \sum_{Q \in \wt \Tree(Q_i^j)} \beta_{E}^{d,2}(3B_Q) \ell(Q)^d \leq C(c_1, \tau, n, C_0) \ell(Q_i^j)^d. 
\end{align}
\end{lemma}
To prove Lemma \ref{l:Tree-bound}, we will need the following Lemma from \cite{AS18}.
\begin{lemma}[{\cite{AS18}, Lemma}] \label{l:AS-beta}
Let $1 \leq p < \infty$ and $E_1$, $E_2$ lower content $d$-regular subsets of $\R^n$; let moreover $x \in E_1$ and choose a radius $r>0$. Then if $y \in E_2$ is so that $B(x,r) \subset B(y, 2r)$, we have
\begin{align}
    \beta_{E_1}^{p,d}(x,r) \lesssim \beta_{E_2}^{p,d}(y, 2r) + \ps{\frac{1}{r^d} \int_{E_1 \cap B(x,2r)} \ps{\frac{\dist(y, E_2)}{r}}^p \, d \hdc(y)}^{\frac{1}{p}}. 
\end{align}
\end{lemma}

Let $R, j, i$ as above. Let $E_{Q_i^j} = E_{i,j}$ be the Ahlfors regular set obtained from the Main Lemma for $\wt\Tree(Q_i^j)$ and $d_{Q_i^j}$ be the function defined in \eqref{e:d_F}, where, in this instance, $\cF = \wt \Stop(Q_i^j)$ and $\cT = \wt \Tree(Q_i^j)$. Specifically, for $C_0>4$, as in \eqref{e:hatT}, we set
\begin{align*}
    \hat T_{i,j} := \ck{Q \in \cD \, |\, \ell(Q)= \ell(Q_i^j), \, Q \cap C_0 B_{Q_i^j} \neq \emptyset};
\end{align*}
following \eqref{e:cC}, we then put
\begin{align*}
    \cC_{i,j} :=\ck{ I \in \cI \, |\, I \cap \hat T \neq \emptyset \mbox{ and } I \mbox{ is maximal with } \ell(I) < \tau d_{Q_i^j}(I)}.
\end{align*}
Then
\begin{align}
 E_{i,j} := \bigcup_{I \in \cC_{i,j}} \partial_d I.
\end{align}
It follows from the Main Lemma that $E_{i,j}$ is Ahlfors $d$-regular. 

\begin{lemma}\label{l:EQ-UR}
Let $k_0, \tau>0, R, j$ and $i$ as above. Then $E_{i,j}$ is uniformly rectifiable.
\end{lemma}

We want to use the fact that $\sP$ guarantees uniform rectifiability and that it is continuous. We will show that there exist constants 
\begin{align*}
    c_1,...,c_N
\end{align*}
such that the measure 
\begin{align}
    \chara_{\cB^{\sP}(c_1p_1,...,c_Np_N)}(x,r) \dt d\hd(x) 
\end{align}
is Carleson on $E_{i,j} \times \R_+$. We test this measure on a ball $B$ centered on $E_{i,j}$ and with radius $r_B$. Note that
\begin{align} \label{e:carleson1}
    \int_{B \cap E_{i,j}} \int_0^{\eta \tau d_{Q_i^j}(x)} \chara_{\cB^{\sP}(c_1p_1,...,c_Np_N)} (x,r)\, \dr d\hd(x) \lesssim_{n,d}\, r_B^d. 
\end{align}
holds automatically: indeed, for any $x \in E_{i,j}$ and whenever $0<r \leq \eta \tau d_{Q_i^j}(x)$, $B(x, r) \cap E_{i,j}$ is just a finite union of $d$-dimensional planes, and the number of planes in this union is bounded above by a universal constant only depending on $n$ and $d$. Therefore $B(x, r) \cap E_{i,j}$ is uniformly rectifiable and thus \eqref{e:carleson1} holds. 
Also, using the Ahlfors regularity of $E_{i,j}$, it is immediate to see that
\begin{align} \label{e:carleson2}
    \int_{B \cap E_{i,j}} \int_{\eta \tau d_{Q_i^j}(x)}^{\eta^{-1} d_{Q_i^j}(x)}
\chara_{\cB^{\sP}(c_1p_1,...,c_Np_N)} (x,r)\, \dr d\hd(x) \lec_{\tau, \eta} \, r_B^d. 
\end{align}

Let us check that
\begin{align}\label{e:carleson3}
\int_{B \cap E_{i,j}} \int_{\eta^{-1}d_{Q_i^j}(x)}^{\frac{\tau}{10} \ell(Q_i^j)}
\chara_{\cB^{\sP}(c_1p_1,...,c_Np_N)} (x,r)\, \dr d\hd(x) \lec_{\tau, \eta}  \, r_B^d. 
\end{align}

\begin{lemma}\label{l:P-in-ball}
Let $(x,r) \in E_{i,j} \times \R_+$ be such that 
\begin{align}\label{e:tree-assumption}
    \eta^{-1} d_{Q_i^j}(x) \leq r \leq \tau \ell(Q_i^j).
\end{align}
Then, for $\eta>0$ sufficiently small (depending only on $n$), there exists a cube $P$ in $\wt \Tree(Q_i^j)$ so that
\begin{align*}
 B_P  \subset B(x,r).
\end{align*}
\end{lemma}
\begin{proof}
For this proof, we put $Q= Q_i^j$.
Let $I_x$ be the cube in $\cC_{Q}$ containing $x$,  so $\ell(I_{x})\sim \tau d_{Q}(x)$. Let $P^*$ be the minimiser of $d_{Q}(x)$.  Note that 
\begin{equation}
    \label{e:distp*}
\dist(x,P^*)\leq d_{Q}(x)\leq \eta r.
\end{equation}

Let us look at two distinct cases. 

\textbf{Case 1.} Suppose first that 
\begin{align}\label{e:tree-case1}
    d_{Q}(x) = \ell(P^*) + \dist(x, P^*) \leq 2 \ell(P^*).
\end{align}
Then we immediately obtain that
\begin{align*}
    \ell(P^*) \leq d_Q(x) \leq 2 \ell(P^*)
\end{align*}
and therefore that
\begin{align}\label{e:tree-case1-b}
    \ell(P^*)\sim \tau^{-1}\ell(I_x).  
\end{align}
But \eqref{e:tree-case1} also implies that \begin{align} \label{e:tree-case1-c}
    \dist(x, P^*) \leq \ell(P^*)
\end{align}
Now, because of the assumption \eqref{e:tree-assumption}, we see that (using also \eqref{e:tree-case1-b}) \begin{align*}
    r \geq \eta^{-1} d_Q(x) \geq \eta^{-1}d_Q(I_x) \sim \eta^{-1} \tau^{-1} \ell(I_x) \sim \eta^{-1} \ell(P^*),
\end{align*}
and so, because \eqref{e:tree-case1-c} and \eqref{e:distp*}, we have  for $\eta$ small $B_{P^*} \subset B(x,r)$.

\textbf{Case 2.} Suppose now that 
\begin{align*}
    d_Q(x) = \ell(P^*) + \dist(x, P^*) \leq 2 \dist(x, P^*). 
\end{align*}
Then we have
\begin{align*}
    \dist(x, P^*) \sim d_Q(x) \leq C \eta r.
\end{align*}
Also, by \eqref{e:tree-assumption}, it holds that 
\begin{align*}
    \ell(P^*) \leq d_Q(x) \leq \eta r.
\end{align*}
This implies, for $\eta>0$ sufficiently small,  that also in this case we have $B_{P^*} \subset B(x, r)$.
\end{proof}

\begin{lemma}
There exist constants $(c_1,...,c_N)$ such that the following holds. 
Let $(x,r) \in E_{i,j} \times \R_+$ be such that
\begin{align}\label{e:tre-assumption2}
    \eta^{-1} d_{Q_i^j}(x) \leq r \leq \tau \ell(Q_i^j).
\end{align}
Then
\begin{align*}
    (x,r) \in \cG_{E_{i,j}}^\sP(c_1p_1,...,c_Np_N). 
\end{align*}
\end{lemma}

\begin{proof}
We know from Lemma \ref{l:P-in-ball} that if $(x,r)$ satisfies \eqref{e:tre-assumption2}, then there exists a cube $P^* \in \wt \Stop(Q_i^j)$ such that $B_{P^*} \subset B(x,r)$. Thus, there must exist an ancestor $\widehat P^* \in \wt \Tree(Q_i^j)$ of $P^*$ so that 
\begin{align} \label{e:r-ell}
     \rho\ell(\widehat P^*)  \leq 4C_2 r <  \ell(\widehat P^*),
\end{align}
and thus so that $B(x,C_2 r) \subset B_{\wh P^*}$, and since $C_1>4C_2/\rho$, we also have $r\geq \ell(\wh P^{*})/C_1$. But recall that if $\widehat P^* \in \wt \Tree(Q_i^j)$, then we must have, by definition, that $(\zeta_{\widehat P^*}, \ell(\widehat P^*)) \in \cG^\sP(p_1,...,p_N)$.


Let us check that
\begin{align*}
    d_{C_2 B_{\wh P^*}} (E_{i,j}, E) < \tau.
\end{align*}

  By \eqref{e:adr-corona}, if $y \in E \cap C_2 B_{\wh P^*}$
\[
\dist(y,E_{i,j})\lec \tau d_{Q_{i,j}}(y)\leq \tau (\dist(y,\wh P^*)+\ell(\wh P^*))\lec C_2\tau \ell(\wh P^*).
\]


That for any $x \in E_{i,j} \cap C_2B_{\wh P^*}$ we have $\dist(x,E) \lec \tau \ell(\wh P^*)$ follows in the same way, since any such $x$ is contained in a dyadic cube $I$ touching $E$ so that 
\[
\ell(I)<\tau d_{Q_i^j}(I) 
\stackrel{\eqref{e:tre-assumption2}}{\leq} \eta \tau r  
\stackrel{\eqref{e:r-ell}}{\leq} 8\eta\tau \ell(\wh P^*).\]

Choosing $\tau$ in the construction of $\cC_{i,j}$ appropriately (depending on $\ve$ and $C_{2}$),  the lemma follows from the $(\ve, C_1,C_2)$-continuity of $\sP$.
\end{proof}

\begin{proof}[Proof of Lemma \ref{l:EQ-UR}] 
 We have shown that there exist constants $c_1,...,c_N$ such that, for any pair $(x,r) \in E_{i,j} \times \R_+$ with 
 \begin{align*}
     \eta^{-1} d_{Q_i^j}(x) \leq r \leq \tau \ell(Q_i^j)
 \end{align*}
 we have
\begin{align}
    (x, r) \in \cG^{\sP}_{E_{i,j}} (c_1p_1,...,c_Np_N). 
\end{align}
Thus the integral in \eqref{e:carleson3} equals to zero. 
Now, we also see that, trivially
\begin{align}
    \int_{B \cap E_{i,j}} \int_{\frac{\tau}{10} \ell(Q_i^j)}^{\diam(E_{i,j})}
\chara_{\cB^{\sP}(c_1p_1,...,c_Np_N)} (x,r)\, \dr d\hd(x) \lec_{\tau}\, r_B^d.
\end{align}
This together with the previous estimates  \eqref{e:carleson1}, \eqref{e:carleson2} and \eqref{e:carleson3} proves that the measure $\chara_{\cB^\sP(c_1p_1,...,c_Np_N)}(x,r) \dr d \hd|_{E_{i,j}}(x)$ is a Carleson measure on $E_{i,j} \times \R_+$; then, because $\sP$ guarantees uniform rectifiability with the appropriate parameters and it is $(\ve, C_1,C_2)$-continuous, $E_{i,j}$ is uniformly rectifiable. Note that all the constants involved depend only on $n,d, \tau, \eta$ (and $c_0$); in particular, they are all independent of $Q_i^j$, $R$ and $k_0$.
\end{proof}

\begin{proof}[Proof of Lemma \ref{l:Tree-bound}]
We want to apply Lemma \ref{l:AS-beta} with $E_1=E$, $E_2= E_{i,j}$ and $p=2$. For $Q \in \wt \Tree(Q_i^j)$, recall that $\zeta_Q$ denotes the center of $Q$. By \eqref{e:adr-corona}, we know that $\dist(z_Q, E_{i,j}) \lesssim \tau d_{Q_i^j} (z_Q) \leq  \tau \ell(Q)$, and in particular, if we denote by $x_Q'$ the point in $E_{i,j}$ which is closest to $x_Q$, we see that $B_Q:=B(\zeta_Q, \ell(Q)) \subset B(x_Q', 2\ell(Q))=:B_Q'$ for $\tau$ small enough. Hence for each cube $Q \in \Tree(Q_i^j)$ the hypotheses of Lemma \ref{l:AS-beta} are satisfied and we may write 
\begin{align*}
   & \sum_{Q \in \wt \Tree(Q_i^j)} \beta_{E}^{2,d}(3 B_Q)^2 \ell(Q)^d \lesssim \sum_{Q \in \wt \Tree(Q_i^j)} \beta_{E_{i,j}}^{2,d} (6 B_Q')^2 \ell(Q)^d \\
    &  \enskip \enskip + \sum_{Q \in \wt \Tree(Q_i^j)} \ps{ \frac{1}{\ell(Q)^d} \int_{6 B_Q \cap E} \ps{\frac{\dist(x, E_{i,j})}{\ell(Q)}}^2 \, d\hdc(x) } := I_1 + I_2. 
\end{align*}
We first look at $I_1$. We apply Theorem \ref{theorem:Christ_cubes} to $E_{i,j}$; let us denote the cubes so obtained by $\cD_{E_{i,j}}$. Note that for each $P \in \wt \Tree(Q^j_i)$ with $P \in \cD(k_0)$, $x_P'$ belongs to some cube $ P' \in \cD_{E_{i,j}}$ so that $\ell(P') \sim \ell(P)$; hence there exists a constant $C_1 \geq 1$ so that 
\begin{align}
    6B'_P \subset C_1B_{P'}.
\end{align}
This in turn implies that $\beta_{E_{i,j}}^{p,d}(6 B'_P) \lesssim_{p,n,d,C_1} \beta_{E_{i,j}}^{p,d}(C_1 B_{P'})$. Hence, 
\begin{align}
    \sum_{\substack{ P \in \Tree(Q_i^j) \\ P \in \cD(k_0)}} \beta_{E_{i,j}}^{2,d} (6 B'_P)^2 \ell(P)^d \lesssim_{p,n,d, C_1} \sum_{\substack{P' \in \cD_{E_{i,j}} \\ \ell(P') \lesssim \ell(Q_i^j) }} \beta_{E_{i,j}}^{2,d} (C_1 B_{P'})^2 \ell(P')^d. 
\end{align}
Since $E_{i,j}$ is uniformly rectifiable, we immediately have that $I_1 \lesssim \ell(Q_i^j)^d$ by the main results of \cite{DS} (in particular, see (C3) and (C6) in \cite[Chapter 1]{DS}).

Let us now worry about $I_2$. We put
\begin{align} 
     \Approx(Q_i^j) 
     := \Big\{ \mbox{ maximal } S \in \cD(k_0) \, |\, & \mbox{there is } I \in \cC_{i,j} \mbox{ s.t. } I \cap S \neq \emptyset \nonumber\\
   &  \mbox{ and } \rho \ell(S) \leq  \ell(I) \leq \ell(S)\Big\}. \label{e:Approx}
\end{align}
It is clear that 
\begin{align}\label{e:approx-cont}
    Q_i^j \subset \bigcup_{S \in \Approx(Q_i^j)} S, 
\end{align}
since $\cC_{i,j}$ covers $Q_i^j$. Now let $x \in Q_i^j$. We claim that there exists a cube $S \in \Approx(Q_i^j)$ so that 
\begin{align}\label{e:Approx-dist}
    \dist(x, E_{i,j}) \leq C \ell(S).
\end{align}
By \eqref{e:approx-cont}, we see that if $x \in Q_i^j$, then there exists an $S \in \Approx(Q_i^j)$ so that $x \in S$. But, then, by definition, there exists an $I \in \cC_{i,j}$ such that $\ell(I) \leq \ell(S)$ and $I \cap S \neq \emptyset$. Thus 
\[
\dist(x,E_{i,j})\leq \diam I + \dist(x, I) \lec  \ell(S).\]
We now estimate $I_2$ as follows: first, 
\begin{align}
   \frac{1}{\ell(Q)^d} \int_{6 B_Q \cap E} \ps{\frac{\dist(x, E_{i,j})}{\ell(Q)}}^2 \, d\hdc(x) & \lesssimt{\eqref{e:Approx-dist}} \hspace{-10pt}   \sum_{ \substack{S \in \Approx(Q_i^j) \\ S \cap 6 B_Q \neq \emptyset}} \int_S \frac{\ell(S)^2}{\ell(Q)^{d+2}} \, d\hdc \nonumber \\
    & \lesssim \sum_{ \substack{S \in \Approx(Q_i^j) \\ S \cap 6 B_Q \neq \emptyset}}\frac{\ell(S)^{2+d}}{\ell(Q)^{2+d}}. \label{e:beta-est5}
\end{align}
Hence we obtain that 
\begin{align} \label{e:beta-est1}
    & I_2 \lesssimt{\eqref{e:beta-est5}} \sum_{Q \in \wt  \Tree(Q_i^j)} \sum_{ \substack{S \in \Approx(Q_i^j) \\ S \cap 6 B_Q \neq \emptyset}}\frac{\ell(S)^{d+2}}{\ell(Q)^{2}} \nonumber \\
    & \lesssim \sum_{\substack{S \in \Approx{Q_i^j} \\ S \cap 6 B_{Q_i^j}}}\ell(S)^{d+2} \sum_{\substack{Q \in \wt \Tree(Q_i^j) \\ S \cap 6 B_Q \neq \emptyset } } \frac{1}{\ell(Q)^{2}}.
\end{align}
Note that the number of cubes $Q \in \wt \Tree(Q_i^j)$ which belong to a given generation and such that $S \cap 6 B_Q \neq \emptyset$ is bounded above by a constant $C$ which depends on $n$. Indeed, if $S \cap 6B_Q \neq \emptyset $, then we must have that $\dist(Q, S) \leq 6\ell(Q)$. Moreover, because $S \in \Approx(Q_i^j)$ and using Lemma \ref{lemma:SS_prop2}, we see that, if $I \in \cC_{i,j}$ is so that $I \cap S \neq \emptyset$ and $\ell(S) \sim\ell(I)$ (as in \eqref{e:Approx}), 
\[
    \ell(S) \sim  \ell(I) \sim{\tau} d_{Q_i^j}(I) \lesssim 
\tau(    \ell(Q) + \dist(I,Q)) \leq \tau \ell(Q)+\tau(6\ell(Q)+2\ell(S))
\]
so for $\tau$ small enough, $\ell(S)\lec \ell(Q)$. Thus we can sum the interior sum in \eqref{e:beta-est1}:
\begin{align*}
    \sum_{\substack{ Q \in \Tree(Q_i^j) \\ S \cap 6B_Q \neq \emptyset }} \frac{1}{\ell(Q)^{2}} \lesssim_{\tau} \frac{1}{\ell(S)^{2}}.
\end{align*}
Finally, we see that 
\begin{align} \label{e:beta-est2}
    I_2 \lesssim_{\tau, n} \sum_{ \substack{S \in \Approx(Q_i^j) \\ S \cap 6 B_{Q_i^j}}} \frac{\ell(S)^{d+2}}{\ell(S)^{2}} = \sum_{ \substack{S \in \Approx(Q_i^j) \\ S \cap 6 B_{Q_i^j}}} \ell(S)^d.
\end{align}
Now, by definition of $\Approx(Q_i^j)$, the last sum in  \eqref{e:beta-est2} is bounded above by a constant times
\begin{align*}
    \sum_{I \in \cC_{i,j}} \ell(I^d) \lec \hd \ps{\bigcup_{I \in \cC_{i,j}} \partial_d I} = \hd(E_{i,j}) \lesssim \ell(Q_i^j)^d,
\end{align*}
where we also used the Ahlfors regularity of $E_{i,j}$. 
This proves \eqref{e:beta-bound-tree}.
\end{proof}

\begin{proof}[Proof of Lemma \ref{l:meta}]
Let $Q_0 \in \cD$ as in the statement of the Lemma. Then we see that
\begin{align} \label{e:beta-est3}
    & \sum_{R \in \Top(k_0)}\sum_{j=0}^{K_R} \sum_{Q \in \Next_j(R)} \sum_{P \in \wt \Tree(Q)} \beta_E^{2, d}(3 B_P)^2 \ell(P)^d \nonumber \\
    &  \stackrel{\eqref{e:beta-bound-tree}}{\lec_{\tau}} \sum_{R \in \Top(k_0)} \sum_{j=0}^{K_R} \sum_{Q \in \Next_j(R)} \ell(Q)^d.
\end{align}
Note that for $1 \leq j \leq K_R$, if $Q \in \Next_j(R)$, then there is a sibling $Q'$ of $Q$ so that $(\zeta_{Q'}, \ell(Q')) \in \cB^{\sP}$. Also recall that we put $\Next_0(R)= \{R\}$. Then any cube appearing in the sum \eqref{e:beta-est3}, either belongs to $\Top(k_0)$ (whenever it belongs to $\Next_0(R)$), or is adjacent to a cube in $\cB^\sP(Q_0,p_1,...,p_N)$, as defined in \eqref{e:cube-B-G}. Thus we see that 
\[
    \eqref{e:beta-est3}  \lec_{\tau} \ps{ \sum_{R \in \Top(k_0)} \ell(R)^d + \sP(Q_0, p_1,...,p_N)} 
     \stackrel{\eqref{e:ADR-packing}}{\lec_{\tau}} \, \hd(Q_0) + \sP(Q_0).
\]
Note that all these estimates were independent of $k_0$. Sending $k_0$ to infinity and recalling \eqref{e:allcomparable} (and recalling that $\ell(Q_{0})^{d}\lec_{c} \cH^{d}(Q_0)$) gives the estimate \eqref{e:meta-beta-est}.

\end{proof}

\section{Applications: The dimensionless quantities LS and LCV}
\label{s:LS}

Here we give a proof of Theorem \ref{t:LS}

\begin{proof}
First, it is not hard to show that there is $c>0$ so that if $Q\in \good{BWGL}_{E}(Q_{0},c\ve)$, then for any children $Q'$ of $Q$, since 
\[\ell(Q')=\rho \ell(Q)<\frac{1}{4} \ell(Q),\] 
we have $Q'\in \good{LS}_{E}(\ve)$. Using this fact, we get
\[
\cH^{d}(R)+{\rm LS}(R,\ve)
\leq \cH^{d}(R)+{\rm BWGL}(R,c\ve)
\lec \beta_{E}(R)\]
and so we just need to prove the reverse inequality.

First we show that for all $C>1$ and $\ve>0$ is small depending on $C$ and $B\in \good{LS}_{E}(\ve)$ and $E'$ is another lower $d$-regular set so that $d_{4B}(E,E')<\ve$, then any ball $B'$ with $4B' \subseteq B$ centered on $E'$ with $r_{B'}\geq r_{B}/C$, we have that $B'\in \good{LS}_{E'}(c\ve)$ for some $c>0$, and so LS is $(\ve,C,4)$-continuous for all $C>1$ and $\ve>0$ sufficiently small depending on $C$.

Let $x',y'\in E'\cap  B'$, then there are $x,y\in E$ with $|x-x'|,|y-y'|<4\ve r_{B}$. For $\ve>0$ small depending on $C$, since $r_{B'}\geq r_{B}/C$,  $x,y\in \frac{3}{2}B'$, and so $2x-y\in 3B'\subseteq B$. Since $B\in \cG^{{\rm LS}}(\ve)$, there is $\xi\in E$ so that $|2x-y-\xi|<\ve r_{B}$. For $\ve>0$ small enough, since $2x-y\in \frac{3}{2}B'$, $\xi\in 4B'\subseteq B$, thus there is $\xi'\in E'$ with $|\xi-\xi'|<4\ve r_{B}$. Thus,
\[
\dist(2x'-y',E')
\leq |2x'-y'-\xi'|
\leq |2x-y-\xi|+|x-x'|+|y-y'|+|\xi-\xi'|
<16\ve r_{B}.
\]

Hence, $B'\in \cG^{{\rm LS}}_{E'}(16\ve)$. Thus, for $\ve>0$ small enough, Lemma \ref{l:meta} implies the second half of \eqref{e:LS-ineq}. This completes the proof. 

\end{proof}

Another dimensionless quantity is the LCV. This can be proven in much the same way, so we omit the proof.

\begin{theorem} Let $E\subseteq \R^{n}$ be a lower $d$-regular set and $\cD$ its Christ-David cubes. Then for $\ve>0$ small enough, and $R\in \cD$,
\begin{equation}
    \label{e:LS-ineq}
\beta_{E}(R) \sim \cH^{d}(R)+{\rm LCV}(R,\ve).
\end{equation}
\end{theorem}

\section{Application: the BAUP}
\label{s:BAUP}
In this section, we show that we can apply Lemma \ref{l:meta} to the quantitative property BAUP (recall the definition \eqref{e:BAUP-def}). Namely, we will show that BAUP is $(\ve, C_1,C_2)$-continuous. That  BAUP guarantees rectifiability is due to David and Semmes, see \cite{of-and-on}, Proposition 3.18.

Let $\epsilon_0>0$ and $C_0\geq 1$ be given. Let us first define the actual partition that BAUP determines. We put 
\begin{align*}
&
\begin{aligned}[t]
\cG^{\text{BAUP}}(\ve_0, C_0)= \good{BAUP} :=\Bigg\{
      (x,r) \in E \times R_+ \, |&\, \mbox{there is a family } \dF \mbox{ of } d\mbox{-planes} \\
     & \mbox{ s.t. } d_{B(x,C_0 r)} (E, \cup_{P \in \dF} P) < \epsilon_0 \Bigg\}
     \end{aligned}
     \\
    & \cB^{\text{BAUP}}(\ve_0, C_0)  = \bad{BAUP} := E \times \R_+ \setminus \cG^{\text{BAUP}}.
\end{align*}
\begin{lemma}
Let $\ve_0>0$, $C_0\geq 1$, and consider the quantitative property BAUP with parameters $(\ve_0, C_0)$. If $C_{1}\geq 1$, $C_2> 2C_{0}$, $\ve_0$ is small enough (depending on $C_2$ and $C_1$), and $0<\ve_{1}\leq \ve_0$ then BAUP is $(\ve_1, C_1, C_2)$-continuous.
\begin{proof}
Let us consider two subsets $E_1, E_2$ or $\R^n$.
From Definition \ref{d:haus-cont}, we take a ball $B=B(x_B, r_B)$ centered on $E_2$ and so that, first, 
\[
(x_B, r_B) \in \good{BAUP}_{E_1}(\ve_0, C_0),
\]
and second, 
\begin{align} \label{e:E1-E2-dist}
    d_{C_2 B}(E_1,E_2) < \ve_1,
\end{align}
where $C_2$ and $\ve_1\leq \ve_0$ will be determined later with respect to $C_0$ and $\ve_0$. Thus, there is a union of $d$-dimensional planes $\dF$ so that 
\[
d_{C_{0}B}(E_{1},\dF)<\epsilon_{0}.
\]

Next, we consider a ball $B'=B(x_B', r_B')$ centered this time on $E_2$ with $C_2B'\subseteq B$ and so that $r_B' \geq \frac{r_B}{C_1}$. We want to show that for any such a ball $B'$,
\begin{align} \label{e:E2-F-dist}
    d_{ C_0 B' }(E_2, \dF) < c_1 \ve_0 r_B'.
\end{align}
for some constant $c_1$ to be determined.
Let  $y \in E_2 \cap C_0 B'$. Since $2B'\subseteq C_2B' \subset B$, we have $2C_0 B'\subseteq C_0B\subseteq C_2 B$, so we can use \eqref{e:E1-E2-dist} to find an $x \in E_1$ so that $|x-y| < \ve_1 C_2 r_B$. Since $\ve_0\leq 1 $, $x\in E_{1}\cap 2C_0B'\subseteq C_0B$, and because $(x_B, r_B) \in \good{BAUP}_{E_1}(\ve_0, C_0)$, it holds that $\dist(x, \dF) < \epsilon_0 C_0 r_B$. Now, because $\ve_1 \leq \ve_0$, we have that
\begin{align*}
    \sup_{y \in E_2 \cap C_0 B'} \dist(y, \dF) \leq \ve_1 C_2 r_B + \ve_0 C_0 r_B \leq \ps{2 C_2 C_1} \ve_0 r_B'
\end{align*}

Next, for $q \in \dF \cap C_0 B'$, we look at $\dist(q, E_2)$; note in particular that $q  \in \dF \cap C_0 B$  and thus, because $d_{C_0B} (E_1, \dF) < \ve_0$, there is an $x \in E_1$ with $|x-q| \leq \ve_0 C_0 r_B$. Moreover, choosing $C_2>2C_0$, since $\ve_{0}\leq 1$, we also have that $x \in 2C_{0}B \subseteq C_2 B$, and thus $\dist(x, E_2) < C_2 \ve_1 r_B$. All in all, we obtain that
\begin{align*}
    \sup_{q \in E_2 \cap C_0 B'} \dist(q, E_2)&  \leq |x-q| + \dist(x, E_2) \\
    & \leq C_0 \ve_0 r_B + C_2 \ve_1 r_B \\
    & \leq \ps{ 2 C_2 C_1} \ve_0 r_B'.
\end{align*}
This implies \eqref{e:E2-F-dist} with  $c_1= 2C_1C_2$; thus BAUP is $(\ve_1, C_1, C_2)$-continuous, whenever $\epsilon_1\leq \ve_0$, and $C_2$ is sufficiently large, with respect to the parameter $C_0$.
\end{proof}
\end{lemma}

We can now prove Theorem \ref{t:baup}. Firstly, note that we immediately have 
\[
{\rm BAUP}(Q_0,C_0,\ve)\leq  
{\rm BWGL}(Q_0,C_0,\ve) \lec \beta_{E}(Q_0).
\]
Furthermore, since ${\rm BAUP}(C_{0},\ve)$ guarantees and is guaranteed by UR for all $\ve>0$ sufficiently small depending on $C_0$ by \cite[Theorem III.3.18]{of-and-on}. Since it is also $(\ve,C_{1},C_{2})$-continuous for $C_{2}>2C_{0}$ and all $C_{1}\geq 1$ and $\ve>0$ sufficiently small, we have, for  all $C_{0}\geq 1$ and $\ve$ small enough (depending on $C_{0}$)
\[
\beta_{E}(Q_0) \lec \cH^{d}(Q_{0})+ {\rm BAUP}(Q_0,C_0,\ve).
\]
%
%

\section{Application: the GWEC}
Let us give one last example of quantitative property which can be handled within the framework of Lemma \ref{l:meta}. 
For a parameter $\ve_0 >0$, we put in $\bad{GWEC}$ all the pairs $(x, r) \in E \times R_+$ for which there exists an $(n-d-1)$-dimensional sphere $S$ satisfying the following three conditions.
\begin{align}
    & S \subset B(x, r) \mbox{ and } \dist(S, E) > \epsilon_0 r; \label{e:gwec1} \\
    & S \mbox{ can be contracted to a point inside } \nonumber \\
    & \enskip \enskip \enskip \enskip  \ck{ y \in B(x,r) \, |\, \dist(y, E) > \epsilon_0 r }; \label{e:gwec2} \\
    & \mbox{ch}(S) \cap E \neq \emptyset, \label{e:gwec3}
\end{align}
where $\mbox{ch}(S)$ is the convex hull of $S$.
We then put 
\begin{align*}
    \good{GWEC}(\ve_0) := E \times \R_+ \setminus \bad{GWEC}(\ve_0). 
\end{align*}
We want to check that we can apply Lemma \ref{l:meta} with this quantitative property. That the GWEC guarantees uniform rectifiability is Theorem 3.28 in \cite{of-and-on}. All that's left to do is to prove that GWEC is continuous.

\begin{lemma}
The quantitative property GWEC with parameter $\ve_0>0$ is $(\ve_1, C_1, C_2)$-continuous, for all $C_1 \geq 3$, for all $C_2 \geq 1$ and whenever $\ve_1$ is sufficiently small with respect to $\ve_0,C_{1},$ and $C_{2}$. 
\end{lemma}
\begin{proof}
Let $E_1$ and $E_2$ be two subsets of $\R^n$. Let $B=B(x_B, r_B)$ be a ball centered on $E_1$ so that $(x_B, r_B) \in \good{GWEC}_{E_1}(\ve_0)$ and 
\begin{align} \label{e:gwec-E1-E2-dist}
    d_{C_2B} (E_1,E_2) < \ve_1 C_2 r_B.
\end{align}
We want to find a constant $c_1$ so that, for any ball $B'=B(x_B', r_B')$ centered on $E_2$ and with $2B' \subset B$ and $r_B' \geq r_B/C_1$, we have that $(x_B', r_B') \in \good{GWEC}_{E_2}(c_1 \ve_0)$. 

We argue by contradiction. Suppose that for some $c_1$ (to be determined), we can find a sphere $S'$ as in \eqref{e:gwec1}, \eqref{e:gwec2} and \eqref{e:gwec3} for the ball $B'$. We will construct a sphere $S$ for $B$ satisfying the same three conditions: this will contradict the hypothesis that $B$ is a good ball. 

Let $ \hat y \in E_2 \cap \mbox{ch}(S')$; note that in particular $\hat y \in B(x_B', r_B') \subset B$, and thus we can find a point $ \hat x  \in E_1$ with $|\hat y- \hat x| < \ve_1 C_2 r_B$ (using \eqref{e:gwec-E1-E2-dist}). If $W'$ is the $(n-d)$-dimensional plane which contains $S'$, we put $W=W'+(\hat x- \hat y)$. Hence we let $S$ denote the sphere in $W$ with center $\mbox{center}(S') + (\hat x - \hat y)$ and radius equal to that of $S'$.
We claim that $S$ satisfies \eqref{e:gwec1}, \eqref{e:gwec2} and \eqref{e:gwec3} relative to the pair $(x_B, r_B)$. Note first that 
\begin{align}\label{e:S-S'}
    S \subset N_{2 C_2 \ve_1 r_B} (S').
\end{align}
We show that  $\dist(E_1,  N_{2 C_2 \ve_1 r_B} (S')) > \ve_0 r_B$.
Let $s' \in S'$ and $y \in E_1$ be closest to each other. Since $s'\in B$, we must have $y\in 2B$. Let $s\in S$ be closest to $s'$, so $|s-s'|<\epsilon_{1} C_{2} r_{B}$. Let $y'\in E_{2}$ be closest to $y$; then as $y\in 2B$, $|y-y'|<\epsilon_{1} C_{2}r_{B}$; then we have that 
\begin{align*}
    \dist(E_1,  N_{2 C_2 \ve_1 r_B} (S'))&  = |y-s'| \geq |y'-s| - |s'-s| - |y-y'| \\
    & \geq \dist(E_2, S) - 2 \ve_1 C_2 r_B\\
    & \geq c_1 \ve_0 r_B' - 2 \ve_1 C_2 r_B \\
    & \geq \frac{c_1}{C_1} \ve_0 r_B - 2 \ve_1 C_2 r_B .
\end{align*}
Now, choosing $\ve_1$ small enough (depending on $\ve_0$) and $c_1$ sufficiently large (depending on $C_1$), it follows that 
\begin{align*}
    \dist(E_1, S) \geq \dist(E_{1}, N_{2 C_2 \ve_1 r_B} (S'))> \ve_0 r_B.
\end{align*}
This proves \eqref{e:gwec1} for $(x_B, r_B)$. 

We now need to show that we can contract $S$ to a point inside the set $\ck{ y \in  B(x_B,r_B) \, |\, \dist(y, E_1) > \ve_0 r_B}$. To see this, we use \eqref{e:S-S'}: if we denote by $Q_t$ the contraction of $S'$ to a point, then $\dist(Q_{t},E_{2})> c_{1}\epsilon_{0} r_{B}$.  Denote by $\{T_t\}_{0 \leq t \leq 1}$ the homotopy $T_{t}(x) = x+t(\hat{y}-\hat{x})$, so that $T_0(S) = S$, $T_1(S) = S'$ and $T_t(S')$ is a $(n-d-1)$-dimensional sphere lying in the $\mbox{ch}(S \cup S')$. Then we see that $T_t(S) \subset N_{2C_2 \ve_1 r_B}(S')$, so $\dist(T_{t}(S),E_{1})\geq \epsilon_{0}r_{B}$. Thus, putting
\begin{align*}
    \wt T_t(x) := 
    \begin{cases}
    T_{2t}(x) & \mbox{ for } 0 \leq t \leq \frac{1}{2} \\
    Q_{2t-1} & \mbox{ for } \frac{1}{2} \leq t \leq 1,
    \end{cases}
\end{align*}
we see that $\wt T_t$ is the desired contraction; this settles \eqref{e:gwec2}. Moreover, \eqref{e:gwec3} holds from the definition of $S$. But this implies that $(x_B, r_B)$ belongs to $\bad{GWEC}_{E_1}(\ve_0)$. 
This is impossible, and so no sphere $S'$ satisfying $\eqref{e:gwec1}$ to \eqref{e:gwec3} can exists, and therefore $(x_B', r_B') \in \good{GWEC}_{E_2}(c_1 \ve_0)$ for $c_1$ appropriately chosen (depending on $C_1)$, and $\ve_1$ sufficiently small. 
\end{proof}

We can now apply Lemma \ref{l:meta} (and use the fact that $\mathrm{GWEC}(Q_0,\ve)\lec \mathrm{BWGL}(Q_0,c\ve)\lec \beta_{E}(Q_0)$ for some $c>0$), to obtain the following corollary. 
\begin{corollary}
Let $E$ be lower content regular, let $Q_0 \in \cD$. Then 
\begin{align*}
    \beta(Q_0) \sim \hd(Q_0) + \mathrm{GWEC}(Q_0,\ve).
\end{align*}
\end{corollary}
\appendix

\section{The Traveling Salesman Theorem}

In this section we prove Theorem \ref{t:TST}. We begin by recalling the original Traveling Salesman Theorem for higher dimensional sets from \cite[Theorem 3.2 and 3.3]{AS18}.

\begin{theorem}
Let $1\leq d<n$ and $E\subseteq \bR^{n}$ be a closed.
Suppose that $E$ is $(c,d)$-lower content regular and let $\cD$ denote the Christ-David cubes for $E$. Let

\begin{enumerate}
    \item Let $C_{0}>1$ and $A>\max\{C_0,10^5\}>1$, $p\geq 1$, and $\ve>0$ be given. For $R\in \cD$, let
    Then for $R\in \cD$,
\begin{equation}
\label{e:thmiii-upper-bound}
\cH^{d}(R)+  {\rm BWGL}(R,\ve,C_0)
 \lec_{A,n,c, C_0 \ve} \ell(R)^{d}+\sum_{Q\subseteq R} \beta_{E}^{d,p}(AB_Q)^{2}\ell(Q)^{d}.
\end{equation}
Furthermore, if the right hand side of \eqref{e:thmiii-upper-bound} is finite, then $E$ is $d$-rectifiable.
\item For any $A>1$ and $1\leq p<p(d)$, there is $C_{0}\gg A$ and $\ve_{0}=\ve_{0}(n,A,p,c)>0$ such that the following holds. Let  $0<\ve<\ve_{0}$.  Then
\begin{equation}
\ell(R)^{d}+\sum_{Q\subseteq R} \beta_{E}^{d,p}(AB_Q)^{2}\ell(Q)^{d} \sim_{A,n,c, \ve} \cH^{d}(R)+   {\rm BWGL}(R,\ve,C_0)
\label{e:beta<hd2}
\end{equation}
\end{enumerate}
\end{theorem}

If we set 
\[
\beta_{E,A,p}(R):= \ell(R)^{d}+\sum_{Q\subseteq R}\beta_{E}^{d,p}(AB_Q)^{2}\ell(Q)^{d},\]
we will now show
\begin{equation}
    \label{e:allcomparable}
\beta_{E,A,p}(R)\sim_{A,p} \beta_{E,3,2}(R)=: \beta_{E}(R).
\end{equation}

Indeed, one can check that $\beta_{E}^{d,p}(3B_Q)\lec_{A,d,p} \beta_{E}^{d,p}(AB_Q)^{2}$ \cite[Lemma 2.11]{AS18}. Moreover, note that for every $Q\subseteq R$, if $Q^{N}$ denotes the $N$th ancestor of $Q$, then there is $N$ so that $3B_{Q^{N}}\supseteq AB_{Q}$. With these observations, we have 
\[
\beta_{E,3,p}(R)
\lec_{A,p}
\beta_{E,A,p}(R)
\lec_{N} \ell(R)^{d}+ \sum_{Q^{N}\subseteq R}  \beta_{E}^{d,p}(AB_Q)^{2}\ell(Q)^{d}
\lec_{p} \beta_{E,3,p}(R).
\]
Furthermore, by \cite[Lemma 2.13]{AS18}, we see that $\beta_{E}^{d,1}\lec \beta_{E}^{d,p}$ for all $p>1$. Thus, by the Traveling Salesman Theorem, for $A\gg C_0 \gg 3$
\[
\beta_{E,3,p}(R)
\lec {\rm BWGL}(R,\ve,C_0)
\lec \beta_{E,A,1}(R)
\lec \beta_{E,3,1}(R)
\lec \beta_{E,3,2}(R).
\]

This completes the proof of Theorem \ref{t:TST}.

\def\cprime{$'$}

\end{document}